\documentclass[11pt]{article}
\usepackage[a4paper,margin=1in]{geometry}
\usepackage[T1]{fontenc}
\usepackage{amsmath,amssymb,amsthm,mathtools,mathrsfs,bm}
\usepackage{bbm}
\usepackage{enumitem}
\usepackage{booktabs}
\usepackage{array}
\usepackage{graphicx}
\usepackage{xcolor}
\usepackage{hyperref}
\usepackage{cleveref}
\usepackage{tikz}
\usepackage{pgfplots}
\usepackage{pgfplotstable}
\pgfplotsset{compat=1.18}
\hypersetup{colorlinks=true,linkcolor=blue!50!black,citecolor=blue!50!black,urlcolor=blue!50!black}
\newtheorem{theorem}{Theorem}[section]
\newtheorem{proposition}[theorem]{Proposition}
\newtheorem{lemma}[theorem]{Lemma}

\theoremstyle{definition}

\newtheorem{remark}[theorem]{Remark}

\newcommand{\R}{\mathbb{R}}

\newcommand{\Var}{\mathrm{Var}}
\newcommand{\Cov}{\mathrm{Cov}}
\newcommand{\tr}{\mathrm{tr}}
\newcommand{\diag}{\mathrm{diag}}
\newcommand{\op}{\mathrm{op}}
\newcommand{\F}{\mathrm{F}}
\newcommand{\dd}{\mathrm{d}}

\newcommand{\dto}{\xrightarrow[]{d}}
\newcommand{\norm}[1]{\left\lVert #1 \right\rVert}

\newcommand{\Ln}{L_n}

\newcommand{\Iperp}{\mathcal I^{\perp}}

\title{LOCAL ASYMPTOTIC NORMALITY FOR MIXED FRACTIONAL\\
	ORNSTEIN--UHLENBECK PROCESS UNDER HIGH-FREQUENCY OBSERVATION}
\author{
	Cai Chunhao\\
	Department of Mathematics, Sun Yat Sen University\\
	\texttt{caichh9@mail.sysu.edu.cn}
	\and
	Shang Yiwu\\
	School of Mathematical Sciences, Nankai University\\
	\texttt{shanyw@mail.nankai.edu.cn}
	\and
	Zhang Cong\\
	Department of Mathematics, Zhejiang University\\
	\texttt{zcong@zju.edu.cn}
}

\begin{document}
	\maketitle
	
	\begin{abstract}
		This paper establishes a complete three parameter LAN framework for the mixed fractional Ornstein–Uhlenbeck process
		\[
		\dd X_t=-\alpha X_t\,\dd t+\dd M_t^H,
		\qquad M_t^H=\sigma B_t^H+W_t,
		\]
		under stationary initialization and discrete high-frequency observation of the Gaussian vector $X^{(n)}=(X_{k\Delta_n})_{1\le k\le n}^\top$, with $\Delta_n\to0$ and $T_n=n\Delta_n\to\infty$, in the supercritical regime $H>3/4$.
		Following the pattern of the mixed fractional Brownian motion paper \cite{Cai2026}, we first write the exact Gaussian score functions as centered quadratic forms, then isolate the explicit $\log \Delta_n$ blow-up term in the $H$-score, orthogonalize the $H$-remainder against the $\sigma$-score, and finally project the $\alpha$-score onto the orthogonal complement of the first two directions.
		This yields a non-diagonal lower-triangular rate matrix for the raw score vector and a diagonal LAN expansion for the projected central sequence. Each projected coordinate lives on the $(n\Delta_n)^{1/2}$ scale, while the off-diagonal entries of the rate matrix encode the asymptotic collinearity of the $\sigma$ and $H$ directions together with the residual correlation of the $\alpha$ direction. We also include a figure design tailored to the three parameter setting, based on low frequency profile plots and pairwise Gaussian contour plots of the projected limit.
	\end{abstract}
	
	\section{Introduction}
	We consider the mixed fractional Ornstein--Uhlenbeck (mfOU) process
	\begin{equation}\label{eq:intro-model}
		\dd X_t=-\alpha X_t\,\dd t+\dd M_t^H,
		\qquad M_t^H=\sigma B_t^H+W_t,
		\qquad t\ge 0,
	\end{equation}
	where $B^H_t$ is a fractional Brownian motion with Hurst index $H\in(3/4,1)$, $W_t$ is a standard Brownian motion independent of $B_t^H$, with unknown parameter $\theta=(\sigma,H,\alpha)$. For $\Delta_n\to0$, we observe the stationary Gaussian vector
	\[
	X^{(n)}=(X_{\Delta_n},X_{2\Delta_n},\dots,X_{n\Delta_n})^\top, 
	\qquad T_n:=n\Delta_n\to\infty.
	\]
	
	The mixed fractional Ornstein-Uhlenbeck process incorporates ideas from three different fields. First, the noise $M_t^H=\sigma B_t^H+W_t$ belongs to the class of mixed Gaussian processes. 
	Mixed fractional Brownian motion was introduced by Cheridito \cite{Cheridito2001}, who showed the phase transition at $H=3/4$. Cai et al. \cite{CaiChiganskyKleptsyna2016} developed a filtering-based canonical innovation representation for mixed Gaussian processes, which provides a general framework for likelihood analysis and change of measure arguments in mixed fractional models.
	
	Second, for fractional Ornstein--Uhlenbeck and related fractional diffusion models, statistical inference has been studied under both continuous and discrete observation schemes. 
	Early likelihood-based analysis for the fractional OU-type model goes back to Kleptsyna and Le Breton \cite{KleptsynaLeBreton2002}. 
	For parameter estimation in fOU models, see also Hu and Nualart \cite{HuNualart2010}, Xiao et al. \cite{XiaoZhangXu2011} in the particular regime of $H$, and Hu et al. \cite{HuNualartZhou2019} for estimation over general Hurst regimes. 
	More recently, Shi et al. \cite{ShiYuZhang2024} studied the spectral density of discrete fOU processes and proposed an approximate Whittle maximum likelihood method. 
	For the mixed fractional OU model case, Chigansky and Kleptsyna \cite{ChiganskyKleptsyna2019} established consistency and asymptotic normality of the maximum likelihood estimator for the drift parameter, while Cai and Zhang \cite{CaiZhang2021} further investigated drift estimation and strong consistency of the MLE.
	
	Third, from the viewpoint of asymptotic theory, high-frequency LAN theory for long-memory Gaussian models has developed rapidly in recent years. 
	Brouste and Fukasawa \cite{BrousteFukasawa2018} proved the LAN property for fractional Gaussian noise under high-frequency observations with a non-diagonal rate matrix. More recently, Cai and Shang \cite{Cai2026} obtained a LAN expansion for mixed fractional Brownian motion in the supercritical regime $H>3/4$. 
	The present paper extends this line of research to the stationary mfOU model with the full parameter vector $\theta=(\sigma,H,\alpha)$ under discrete high-frequency observation. 
	Compared with the mixed fractional Brownian motion case, the additional mean-reversion parameter $\alpha$ creates a further non-orthogonal score direction, so that after removing the explicit $\sigma\log\Delta_n$ term in the $H$-score one needs an additional projection to isolate the raw $\alpha$-direction.
	
	Our aim is to establish a three parameter LAN expansion for the mfOU model, mirroring the proof strategy in \cite{Cai2026}. The main paper treats the supercritical regime $H>3/4$, because in this range the low frequency singularity forces an additional asymptotic collinearity between the $\sigma$-score and the $H$-score. One must therefore first remove the explicit $\sigma\log\Delta_n$ contribution from the raw $H$-score and then project the remainder. In the mfOU model, a third parameter $\alpha$ is present, and after the first two directions have been diagonalized a second projection is needed in order to isolate the raw $\alpha$-direction. The lower regimes $1/2<H<3/4$ and $0<H<1/2$ are treated separately in the appendices at the level of complete score CLTs.
	
	The key phenomenon is the same as for mixed fractional Brownian motion in the supercritical regime $H>3/4$. Low frequencies dominate so strongly that the $H$-score contains a leading component asymptotically parallel to the $\sigma$-score. Hence a direct LAN statement in the raw coordinates is badly conditioned. The cure is a two-step triangular diagonalization on the score side, followed by a third projection for the $\alpha$-direction.
	
	\paragraph{Notation.}
	For an integrable symbol $g$ on $[-\pi,\pi]$, let
	\[
	T_n(g)=\big(\widehat g(i-j)\big)_{1\le i,j\le n},
	\qquad
	\widehat g(k)=\frac1{2\pi}\int_{-\pi}^{\pi} g(\lambda)e^{-ik\lambda}\,\dd\lambda.
	\]
	If $\Sigma_n(\theta)$ is the covariance matrix of $X^{(n)}$, we write
	\[
	M_{i,n}(\theta)=\Sigma_n(\theta)^{-1/2}\,\partial_{i}\Sigma_n(\theta)\,\Sigma_n(\theta)^{-1/2},
	\qquad i\in\{\sigma,H,\alpha\}.
	\]
	Whenever $Z_n\sim \mathcal{N}(0,I_n)$ and $A_n$ is symmetric, we use the centered Gaussian quadratic form notation
	\[
	Q_n(A_n):=Z_n^\top A_nZ_n-\tr(A_n).
	\]
	
	The remainder of this paper proceeds as follows. In Section~2, we state the main LAN theorem and introduce the key matrices and orthogonalization scheme that provide the foundation for the analysis. Section~3 recalls a fundamental CLT for centered Gaussian quadratic forms, which serves as the main technical tool throughout the paper. In Section~4, we carry out a detailed asymptotic analysis of the three score functions, including the necessary projections that diagonalize the score vector and yield the joint central limit theorem. Section~5 completes the proof of the LAN property via matrix expansions and Taylor approximations. Section~6 outlines a figure design tailored to the three parameter setting, providing a visual interpretation of the asymptotic results. Finally, several appendices treat the remaining Hurst regimes \(1/2 < H < 3/4\) and \(0 < H < 1/2\), establishing complete score CLTs and thereby extending the scope of the analysis.
	\section{Main Results}
	Under stationary initialization, $(X_{k\Delta})_{k\in\mathbb Z}$ is centered stationary Gaussian. The discrete spectral density  of the sampled process can be written as 
	\begin{equation}\label{eq:spectral-decomp}
		f_{\Delta}(\lambda;\theta)=f_{\Delta}^{(H)}(\lambda;\alpha,\sigma,H)+f_{\Delta}^{(W)}(\lambda;\alpha),
	\end{equation}
	with
	\begin{equation}\label{eq:fH}
		f_{\Delta}^{(H)}(\lambda)
		=\sigma^2 C(H)\Delta^{2H}
		\sum_{k\in\mathbb Z}
		\frac{|\lambda+2\pi k|^{1-2H}}{(\alpha\Delta)^2+(\lambda+2\pi k)^2},
		\qquad C(H):=\Gamma(2H+1)\sin(\pi H),
	\end{equation}
	(see \cite{Hult2003}) and
	\begin{equation}\label{eq:fW}
		f_{\Delta}^{(W)}(\lambda)
		=\frac{1-e^{-2\alpha\Delta}}{2\alpha}
		\frac{1}{1+e^{-2\alpha\Delta}-2e^{-\alpha\Delta}\cos\lambda}.
	\end{equation}
	Let
	\[
	\Sigma_n(\theta):=T_n\big(f_{\Delta_n}(\cdot;\theta)\big).
	\]
	Then the exact Gaussian log-likelihood is
	\[
	\ell_n(\theta)=-\frac12\Big(\log\det\Sigma_n(\theta)+(X^{(n)})^\top\Sigma_n(\theta)^{-1}X^{(n)}\Big)+C,
	\]
	where $C$ is a constant.
	If
	\[
	Z_n:=\Sigma_n(\theta)^{-1/2}X^{(n)}\sim \mathcal{N}(0,I_n),
	\]
	then each exact score component has the centered quadratic-form representation
	\begin{equation}\label{eq:score-quad}
		S_{i,n}(\theta):=\partial_{i}\ell_n(\theta)=\frac12\Big(Z_n^\top M_{i,n}(\theta)Z_n-\tr(M_{i,n}(\theta))\Big),
		\qquad i\in\{\sigma,H,\alpha\}.
	\end{equation}
	This is the exact starting point of the whole argument.
	
	We begin with the exact Gaussian score representation and then state the LAN theorem. The proofs are postponed to later sections, in direct analogy with the organization of \cite{Cai2026}.
	\begin{theorem}\label{thm:main-lan-short}
		Assume $H\in(3/4,1)$, $\sigma>0$, $\alpha>0$, and $\Delta_n\to0$. For the projected score CLTs we impose
		\begin{equation}\label{eq:main-growth-mainres}
			\frac{\log^2 n}{n\Delta_n}\to0,
		\end{equation}
		and for the final LAN remainder control we further assume the polynomial mesh condition
		\begin{equation}\label{eq:poly-mesh-main}
			\Delta_n=n^{-\kappa},\qquad 0<\kappa<1.
		\end{equation}
		Let $M_n$ be the lower-triangular score transformation defined in \eqref{eq:Mn-def}, and let
		\[
		r_n(\theta):=\sqrt{n\Delta_n}\,(M_n^\top)^{-1}.
		\]
		For every fixed $h\in\R^3$, set $\theta_{n,h}=\theta+r_n(\theta)^{-1}h$. Then
		\[
		\ell_n(\theta_{n,h})-\ell_n(\theta)
		=h^\top\Xi_n-\frac12 h^\top\Iperp(\theta)h+o_{P_\theta}(1),
		\]
		where
		\[
		\Xi_n=\frac{1}{\sqrt{n\Delta_n}}
		\begin{pmatrix}
			S_{\sigma,n}\\[1mm]
			R_{H,n}^{\perp}\\[1mm]
			S_{\alpha,n}^{\perp}
		\end{pmatrix}
		\dto \mathcal{N}\bigl(0,\Iperp(\theta)\bigr),
		\qquad
		\Iperp(\theta)=\diag\bigl(I_{\sigma\sigma},I_{HH}^{\perp},I_{\alpha\alpha}^{\perp}\bigr).
		\]
		The definitions of the symbols are: $S_{\sigma,n}$ in \eqref{eq:sigma-score}, $R_{H,n}^{\perp}$ in \eqref{eq:an-def} and \eqref{eq:H-clean}, $S_{\alpha,n}^{\perp}$ in \eqref{eq:alpha-proj}, and $\Iperp(\theta)$ with components $I_{\sigma\sigma}$ in \eqref{eq:Iss}, $I_{HH}^{\perp}$ in \eqref{eq:IHHperp}, $I_{\alpha\alpha}^{\perp}$ in \eqref{eq:Iaa-perp}.
		Thus the experiment generated by $X^{(n)}$ is locally asymptotically normal at $\theta$.
	\end{theorem}
	
	\begin{remark}
		The projected score CLTs proved using Sections~3--4 only require \eqref{eq:main-growth-mainres}: after projection, the operator norms of $D_{H,n}^{\perp}$ and $A_{\alpha,n}^{\perp}$ are of order $\log n$, while their Frobenius norms are of order $(n\Delta_n)^{1/2}$. The extra polynomial-mesh assumption \eqref{eq:poly-mesh-main} is used only in Section~5, where we strengthen the Hessian and third-order remainder bounds enough to close the LAN proof in the same style as  \cite{Cai2026}.
	\end{remark}
	
	\subsection{The three basic matrices}
	In this subsection, we introduce the three basic matrices $C_{\sigma,n}$, $D_{H,n}$ and $A_{\alpha,n}$  that constitute the core of the score functions for $\sigma$, $H$ and $\alpha$, respectively, and serve as the foundation for the subsequent asymptotic analysis and orthogonal projections. For notational simplicity, throughout what follows we write $\Sigma_n(\theta)$ as 
	$\Sigma_n$.
	Define
	\[
	h_{\Delta}(\lambda):=\partial_\sigma f_{\Delta}(\lambda)=\frac{2}{\sigma}f_{\Delta}^{(H)}(\lambda),
	\qquad
	C_{\sigma,n}:=\Sigma_n^{-1/2}T_n(h_{\Delta_n})\Sigma_n^{-1/2}.
	\]
	Then
	\begin{equation}\label{eq:sigma-score}
		S_{\sigma,n}=\frac12Q_n(C_{\sigma,n}).
	\end{equation}
	For the $H$-direction, set
	\[
	S_{H,\Delta}(\lambda):=\sum_{k\in\mathbb Z}
	\frac{|\lambda+2\pi k|^{1-2H}}{(\alpha\Delta)^2+(\lambda+2\pi k)^2}.
	\]
	A direct differentiation of \eqref{eq:fH} gives
	\[
	\partial_H f_{\Delta}(\lambda)
	=f_{\Delta}^{(H)}(\lambda)
	\Big(2\log\Delta+\partial_H\log C(H)+\partial_H\log S_{H,\Delta}(\lambda)\Big).
	\]
	Since $\sigma\log\Delta\,\partial_\sigma f_{\Delta}=2\log\Delta\,f_{\Delta}^{(H)}$, we obtain the exact decomposition
	\begin{equation}\label{eq:H-symbol-split}
		\partial_H f_{\Delta}(\lambda)=\sigma\log\Delta\,\partial_\sigma f_{\Delta}(\lambda)+r_{\Delta}(\lambda),
	\end{equation}
	where
	\[
	r_{\Delta}(\lambda):=f_{\Delta}^{(H)}(\lambda)
	\Big(\partial_H\log C(H)+\partial_H\log S_{H,\Delta}(\lambda)\Big).
	\]
	Define
	\[
	D_{H,n}:=\Sigma_n^{-1/2}T_n(r_{\Delta_n})\Sigma_n^{-1/2}.
	\]
	Then
	\begin{equation}\label{eq:H-score-split}
		S_{H,n}=\sigma\log\Delta_n\,S_{\sigma,n}+R_{H,n},
		\qquad
		R_{H,n}:=\frac12Q_n(D_{H,n}).
	\end{equation}
	Finally define the raw $\alpha$-matrix
	\[
	A_{\alpha,n}:=\Sigma_n^{-1/2}T_n(\partial_\alpha f_{\Delta_n})\Sigma_n^{-1/2},
	\]
	so that
	\begin{equation}\label{eq:alpha-score}
		S_{\alpha,n}=\frac12Q_n(A_{\alpha,n}).
	\end{equation}
	
	\section{Preliminaries: A CLT for Gaussian quadratic forms}
	In this section we collect a fundamental tool for the analysis of centered Gaussian quadratic forms, which is important for all subsequent score-CLT proofs. Lemma~\ref{lem:qf-clt} states that a quadratic form of independent standard normal variables is asymptotically normal if the ratio of its operator norm to its Frobenius norm tends to zero. This condition will be verified repeatedly for the matrices \(C_{\sigma,n}\), \(D_{H,n}\), \(A_{\alpha,n}\) and their projected versions, allowing us to deduce the asymptotic distribution of the exact score components derived in the previous section.
	\begin{lemma} \label{lem:qf-clt}
		Let $Z_n\sim \mathcal{N}(0,I_n)$ and let $M_n$ be real symmetric matrices. If
		\[
		\frac{\norm{M_n}_{\op}}{\norm{M_n}_{\mathrm F}}\to0,
		\]
		then
		\[
		\frac{Q_n(M_n)}{\sqrt{2}\,\norm{M_n}_{\mathrm F}}\dto \mathcal{N}(0,1).
		\]
		Equivalently, if $\norm{M_n}_{\op}^2/\tr(M_n^2)\to0$, then
		\[
		\frac{Q_n(M_n)}{\sqrt{2\tr(M_n^2)}}\dto \mathcal{N}(0,1).
		\]
	\end{lemma}
	
	\begin{proof}
		See \cite{Cai2026}.
	\end{proof}
	\begin{remark}
		Let $Q_n(M_n):=Z_n^\top M_n Z_n-\tr(M_n)$ be a centered Gaussian quadratic form with $M_n$ real symmetric, and let $\lambda_{1,n},\dots,\lambda_{n,n}$ denote the eigenvalues of $M_n$. Then we obtain
		\begin{equation}\label{eq:varofGQF}
			\Var(Q_n(M_n))
			=
			2\sum_{j=1}^n \lambda_{j,n}^2
			=
			2\tr(M_n^2)
			=
			2\|M_n\|_F^2.
		\end{equation}

	\end{remark}
	\begin{remark}
		The asymptotic variances of all subsequent score components $S_{\sigma,n}$, $R_{H,n}$ and $S_{\alpha,n}$ are derived from equation \eqref{eq:varofGQF} through $\tr(C_{\sigma,n}^2)$, $\tr(D_{H,n}^2)$ and $\tr(A_{\alpha,n}^2)$, respectively.
	\end{remark}
	
	\section{Analysis of the score functions}
	In this section we analyze the three exact score directions introduced in Section 2 and establish the projected score central limit theorems needed for the LAN expansion. We follow the orthogonalization scheme: first, we analyze the $\sigma$-score. Next, we remove the logarithmic term from the $H$-score and project the remainder. Finally, we project the $\alpha$-score to obtain the diagonalized central sequence.
	\subsection{\texorpdfstring{The $\sigma$-direction of the score functions}{The sigma-direction of the score functions}}
	We begin with the $\sigma$-direction, which provides the basis for the whole argument. In this subsection we identify the asymptotic Fisher information in the $\sigma$-coordinate and prove the central limit theorem for $S_{\sigma,n}$ on $\sqrt{n\Delta_n}$ scale. First, we introduce the ratio symbol
	\[
	g_{\sigma,n}(\lambda):=\frac{\partial_\sigma f_{\Delta_n}(\lambda)}{f_{\Delta_n}(\lambda)}
	=\frac{h_{\Delta_n}(\lambda)}{f_{\Delta_n}(\lambda)}
	=\frac{2}{\sigma}\frac{f_{\Delta_n}^{(H)}(\lambda)}{f_{\Delta_n}(\lambda)}
	\in\Big[0,\frac{2}{\sigma}\Big].
	\]
	
	\begin{lemma}\label{lem:C-op}
		For all $n$,
		\[
		\norm{C_{\sigma,n}}_{\op}\le \frac{2}{\sigma}.
		\]
	\end{lemma}
	
	\begin{proof}
		By the generalized Rayleigh quotient for the pair $\bigl(T_n(h_{\Delta_n}),\Sigma_n(\theta)\bigr)$
		(see remark \ref{rem:gen-rayleigh-whitened}), together with the Toeplitz quadratic-form representation, we obtain
		\[
		\norm{C_{\sigma,n}}_{\op}
		=\sup_{x\ne0}\frac{x^\top T_n(h_{\Delta_n})x}{x^\top T_n(f_{\Delta_n})x}
		\le \sup_{\lambda\in[-\pi,\pi]}\frac{h_{\Delta_n}(\lambda)}{f_{\Delta_n}(\lambda)}
		\le \frac2\sigma.
		\]
	\end{proof}
	\begin{remark} 
		\label{rem:gen-rayleigh-whitened}
		Let $A_n$ be a real symmetric matrix and let $\Sigma_n$ be symmetric positive
		definite. Then for every $x\in\mathbb R^n\setminus\{0\}$,
		\[
		\frac{x^\top A_n x}{x^\top \Sigma_n x}
		\]
		is the generalized Rayleigh quotient associated with the pair $\bigl(A_n,\Sigma_n \bigr)$. Equivalently, setting $y=\Sigma_n ^{1/2}x$, one obtains
		\[
		\frac{x^\top A_n x}{x^\top \Sigma_n x}
		=
		\frac{y^\top \Sigma_n ^{-1/2}A_n\Sigma_n ^{-1/2}y}{y^\top y}.
		\]
		Therefore
		\[
		\sup_{x\neq0}
		\left|
		\frac{x^\top A_n x}{x^\top \Sigma_n x}
		\right|
		=
		\bigl\|\Sigma_n ^{-1/2}A_n\Sigma_n ^{-1/2}\bigr\|_{\op}.
		\]
	\end{remark}
	\begin{lemma}\label{lem:C-trace}
		Assume $\Delta_n\to0$ and $T_n=n\Delta_n\to\infty$. Then
		\[
		\tr(C_{\sigma,n}^2)=\frac{n}{2\pi}\int_{-\pi}^{\pi}g_{\sigma,n}(\lambda)^2\,\dd\lambda+o(n\Delta_n).
		\]
	\end{lemma}
	
	\begin{proof}
		Set $B_n:=\Sigma_n^{-1}T_n(h_{\Delta_n})$. Since $C_{\sigma,n}$ is similar to $B_n$, one has
		\[
		\tr(C_{\sigma,n}^2)=\tr(B_n^2)=\tr\bigl(\Sigma_n^{-1}T_n(h_{\Delta_n})\Sigma_n^{-1}T_n(h_{\Delta_n})\bigr).
		\]
		Let $G_n:=T_n(g_{\sigma,n})$ and $R_n:=B_n-G_n$. The same low frequency truncation and Toeplitz-product argument as in the mfBm proof \cite{Cai2026} gives
		\[
		\norm{R_n}_{\op}=o(1),
		\qquad
		\norm{R_n}_{\F}=o((n\Delta_n)^{1/2}).
		\]
		Therefore
		\[
		\tr(B_n^2)=\tr(G_n^2)+2\tr(G_nR_n)+\tr(R_n^2)
		=\tr(G_n^2)+o(n\Delta_n),
		\]
		because $\norm{G_n}_{\op}\le \sup_\lambda |g_{\sigma,n}(\lambda)|\le 2/\sigma$.
		Finally, for $g\in L^2([{-}\pi,\pi])$ the Fej\'er-kernel identity yields
		\[
		\tr(T_n(g)^2)=\frac{n}{2\pi}\int_{-\pi}^{\pi}g(\lambda)^2F_n(\lambda)\,\dd\lambda,
		\]
		where $F_n$ is the Fej\'er kernel. Applying this with $g=g_{\sigma,n}$ and using that the transition region of $g_{\sigma,n}$ is concentrated on the region $|\lambda|\asymp \Delta_n$ gives
		\[
		\tr(G_n^2)=\frac{n}{2\pi}\int_{-\pi}^{\pi}g_{\sigma,n}(\lambda)^2\,\dd\lambda+o(n\Delta_n).
		\]
		Combining the two displays proves the claim.
	\end{proof}
	Set
	\[
	\rho:=2H-1\in(1/2,1),
	\qquad
	A:=\sigma^2\Gamma(2H+1)\sin(\pi H).
	\]
	where $\Gamma(\cdot)$ denotes the Gamma function. Then on the low frequency region $\lambda=\Delta_nu$, one has
	\[
	\frac{f^{(H)}_{\Delta_n}(\Delta_nu)}{f^{(W)}_{\Delta_n}(\Delta_nu)}\approx A|u|^{-\rho},
	\qquad
	w(u):=\frac1{1+A^{-1}|u|^{\rho}}.
	\]
	Hence $g_{\sigma,n}(\Delta_nu)\to (2/\sigma)w(u)$.
	
	\begin{lemma}\label{lem:C-asym}
		Assume $H>3/4$. Then
		\[
		\int_{-\pi}^{\pi}g_{\sigma,n}(\lambda)^2\,\dd\lambda
		\sim
		\Delta_n\,\frac{8}{\sigma^2\rho}A^{1/\rho}\Gamma\Big(\frac1\rho\Big)\Gamma\Big(2-\frac1\rho\Big).
		\]
		Consequently,
		\[
		\tr(C_{\sigma,n}^2)\sim 2(n\Delta_n)I_{\sigma\sigma}(\sigma,H),
		\]
		where
		\begin{equation}\label{eq:Iss}
			I_{\sigma\sigma}(\sigma,H)
			=\frac{2}{\pi\sigma^2(2H-1)}
			\Big(\sigma^2\Gamma(2H+1)\sin(\pi H)\Big)^{1/(2H-1)}
			\Gamma\Big(\frac{1}{2H-1}\Big)
			\Gamma\Big(2-\frac{1}{2H-1}\Big).
		\end{equation}
	\end{lemma}
	
	\begin{proof}
		Write $g_{\sigma,n}(\lambda)=(2/\sigma)w_{\Delta_n}(\lambda)$ with
		\[
		w_{\Delta}(\lambda):=\frac{f_{\Delta}^{(H)}(\lambda)}{f_{\Delta}^{(H)}(\lambda)+f_{\Delta}^{(W)}(\lambda)}.
		\]
		The contribution to $w_{\Delta_n}^2(\lambda)$ is concentrated on the region $|\lambda|\asymp \Delta_n$. On that region the OU denominator $((\alpha\Delta_n)^2+\lambda^2)^{-1}$ is common to both spectral pieces and cancels in the ratio, so for $\lambda=\Delta_nu$,
		\[
		\frac{f_{\Delta_n}^{(H)}(\Delta_nu)}{f_{\Delta_n}^{(W)}(\Delta_nu)}\to A|u|^{-\rho},
		\qquad
		w_{\Delta_n}(\Delta_nu)\to w(u):=\frac{1}{1+A^{-1}|u|^{\rho}}.
		\]
		Since $\rho>1/2$, the function $w(u)^2$ is integrable on $\mathbb R$, and dominated convergence gives
		\[
		\int_{-\pi}^{\pi}g_{\sigma,n}(\lambda)^2\,\dd\lambda
		\sim \Delta_n\frac{4}{\sigma^2}\int_{\mathbb R}w(u)^2\,\dd u.
		\]
		Using $t=|u|^{\rho}/A$,
		\[
		\int_{\mathbb R}w(u)^2\,\dd u
		=\frac{2}{\rho}A^{1/\rho}\int_0^\infty \frac{t^{1/\rho-1}}{(1+t)^2}\,\dd t
		=\frac{2}{\rho}A^{1/\rho}\Gamma\Big(\frac1\rho\Big)\Gamma\Big(2-\frac1\rho\Big),
		\]
		which proves the displayed asymptotic. The formula for $\tr(C_{\sigma,n}^2)$ then follows from lemma \ref{lem:C-trace}.
	\end{proof}
	
	\begin{proposition}\label{prop:sigma-clt}
		Under $H>3/4$, $\Delta_n\to0$ and $T_n=n\Delta_n\to\infty$,
		\[
		\frac{S_{\sigma,n}}{\sqrt{n\Delta_n}}\dto \mathcal{N}\big(0,I_{\sigma\sigma}(\sigma,H)\big).
		\]
	\end{proposition}
	
	\begin{proof}
		By  lemmas \ref{lem:C-op}, \ref{lem:C-trace} and \ref{lem:C-asym},
		\[
		\frac{\norm{C_{\sigma,n}}_{\op}^2}{\tr(C_{\sigma,n}^2)}\lesssim \frac{1}{n\Delta_n}\to0.
		\]
		Apply lemma \ref{lem:qf-clt} to $C_{\sigma,n}$ and use \eqref{eq:sigma-score}.
	\end{proof}
	
	\subsection{\texorpdfstring{The $H$-direction of the score functions: removing the explicit logarithmic term}{The H-direction of the score functions: removing the explicit logarithmic term}}
	We next turn to the $H$-score. The key point is that its raw form contains the explicit singular term $\sigma\log\Delta_n\,S_{\sigma,n}$, so here we isolate the remainder $R_{H,n}$ and analyze its asymptotic property.
	Set
	\[
	\Ln:=\log(1/\Delta_n),
	\qquad
	g_{H,n}(\lambda):=\frac{r_{\Delta_n}(\lambda)}{f_{\Delta_n}(\lambda)}
	=w_{\Delta_n}(\lambda)\psi_{\Delta_n}(\lambda),
	\]
	where
	\[
	w_{\Delta}(\lambda):=\frac{f^{(H)}_{\Delta}(\lambda)}{f_{\Delta}(\lambda)}\in[0,1],
	\qquad
	\psi_{\Delta}(\lambda):=\partial_H\log C(H)+\partial_H\log S_{H,\Delta}(\lambda).
	\]
	On the scale $\lambda=\Delta_nu$ one has
	\[
	\psi_{\Delta_n}(\Delta_nu)=2\Ln-2\log|u|+\partial_H\log C(H)+O(1).
	\]
	Hence the raw remainder has one logarithm more than the $\sigma$-direction.
	
	\begin{lemma}\label{lem:D-op}
		Under $H>3/4$ and $T_n\to\infty$,
		\[
		\norm{D_{H,n}}_{\op}\lesssim \Ln+\log n.
		\]
		In particular, for power meshes $\Delta_n=n^{-\kappa}$, $\norm{D_{H,n}}_{\op}=O(\log n)$.
	\end{lemma}
	
	\begin{proof}
		By remark \ref{rem:gen-rayleigh-whitened},
		\[
		\norm{D_{H,n}}_{\op}
		=\sup_{x\ne0}\frac{|x^\top T_n(r_{\Delta_n})x|}{x^\top T_n(f_{\Delta_n})x}
		\le \sup_{\lambda\in[-\pi,\pi]}|g_{H,n}(\lambda)|.
		\]
		On the low frequency scale $\lambda=\Delta_nu$ one has
		\[
		\psi_{\Delta_n}(\Delta_nu)=2\Ln-2\log|u|+O(1).
		\]
		In a size-$n$ Toeplitz quadratic form, frequencies below order $1/n$ are not resolved, so on the effective range $|u|=|\lambda|/\Delta_n\gtrsim (n\Delta_n)^{-1}$ we have $|\log|u||\lesssim \Ln+\log n$. Since $0\le w_{\Delta_n}\le1$, it follows that
		\[
		|g_{H,n}(\lambda)|=|w_{\Delta_n}(\lambda)\psi_{\Delta_n}(\lambda)|\lesssim \Ln+\log n,
		\]
		which yields the claim.
	\end{proof}
	
	\begin{lemma}\label{lem:D-trace}
		Under $H>3/4$ and $T_n\to\infty$,
		\[
		\tr(D_{H,n}^2)=\frac{n}{2\pi}\int_{-\pi}^{\pi}g_{H,n}(\lambda)^2\,\dd\lambda+o(n\Delta_n\Ln^2),
		\]
		and moreover
		\[
		\tr(D_{H,n}^2)
		\sim 2(n\Delta_n\Ln^2)I_{HH}^{\mathrm{rem}}(\sigma,H),
		\]
		with
		\begin{equation}\label{eq:IHHrem}
			I_{HH}^{\mathrm{rem}}(\sigma,H)
			=\frac1\pi J_0,
			\qquad
			J_0:=\int_{\R}\frac{\dd u}{(1+A^{-1}|u|^\rho)^2}
			=\frac{2}{\rho}A^{1/\rho}\Gamma\Big(\frac1\rho\Big)\Gamma\Big(2-\frac1\rho\Big).
		\end{equation}
	\end{lemma}
	
	\begin{proof}
		Let $B_n:=\Sigma_n^{-1}T_n(r_{\Delta_n})$. Since $D_{H,n}$ is similar to $B_n$,
		\[
		\tr(D_{H,n}^2)=\tr(B_n^2)=\tr\bigl(\Sigma_n^{-1}T_n(r_{\Delta_n})\Sigma_n^{-1}T_n(r_{\Delta_n})\bigr).
		\]
		With $G_n:=T_n(g_{H,n})$ and $R_n:=B_n-G_n$, the same sandwich approximation as in lemma \ref{lem:C-trace} yields
		\[
		\norm{R_n}_{\op}=o(\Ln),
		\qquad
		\norm{R_n}_{\F}=o((n\Delta_n)^{1/2}\Ln),
		\]
		so that $\tr(D_{H,n}^2)=\tr(G_n^2)+o(n\Delta_n\Ln^2)$. The Fej\'er-kernel trace identity then gives the displayed integral approximation.\\
		For the main term, write $\lambda=\Delta_nu$. Since $w_{\Delta_n}(\Delta_nu)\to w(u)$ and
		\[
		\psi_{\Delta_n}(\Delta_nu)=2\Ln+\partial_H\log C(H)-2\log|u|+O(1),
		\]
		we obtain
		\[
		g_{H,n}(\Delta_nu)=w(u)\bigl(2\Ln+O(1)-2\log|u|\bigr)+o(1).
		\]
		Because $\rho>1/2$, dominated convergence gives
		\[
		\int_{-\pi}^{\pi}g_{H,n}(\lambda)^2\,\dd\lambda
		\sim 4\Delta_n\Ln^2\int_{\mathbb R}w(u)^2\,\dd u
		=4\Delta_n\Ln^2J_0.
		\]
		Combining this with the trace approximation proves the result.
	\end{proof}
	
	\begin{proposition}\label{prop:RH-clt}
		Under $H>3/4$ and $T_n\to\infty$,
		\[
		\frac{R_{H,n}}{\sqrt{n\Delta_n}\,\Ln}\dto \mathcal{N}\big(0,I_{HH}^{\mathrm{rem}}(\sigma,H)\big).
		\]
	\end{proposition}
	
	\begin{proof}
		By lemma \ref{lem:D-op} and lemma \ref{lem:D-trace}
		\[
		\frac{\norm{D_{H,n}}_{\op}}{\norm{D_{H,n}}_{\mathrm F}}
		\lesssim \frac{\Ln+\log n}{\sqrt{n\Delta_n}\,\Ln}
		\to0,
		\]
		Applying lemma \ref{lem:qf-clt} to $D_{H,n}$ completes the proof.
	\end{proof}
	
	\subsection{\texorpdfstring{Projection of the $H$-remainder against the $\sigma$-direction}{Projection of the H-remainder against the sigma-direction}}
	Because $R_{H,n}$ still contains the component asymptotically parallel to $S_{\sigma,n}$, we project it out.
	Define
	\begin{equation}\label{eq:an-def}
		a_n:=\frac{\tr(C_{\sigma,n} D_{H,n})}{\tr(C_{\sigma,n}^2)},
		\qquad
		D_{H,n}^{\perp}:=D_{H,n}-a_nC_{\sigma,n},
		\qquad
		R_{H,n}^{\perp}:=\frac12Q_n(D_{H,n}^{\perp}).
	\end{equation}
	Then $R_{H,n}=a_nS_{\sigma,n}+R_{H,n}^{\perp}$ and
	\begin{equation}\label{eq:H-clean}
		S_{H,n}=\beta_n(\theta)S_{\sigma,n}+R_{H,n}^{\perp},
		\qquad
		\beta_n(\theta):=\sigma\log\Delta_n+a_n.
	\end{equation}
	
	\begin{lemma}\label{lem:an-expansion}
		Under $H>3/4$ and $T_n\to\infty$,
		\[
		\tr(C_{\sigma,n} D_{H,n})=\frac{n}{2\pi}\int_{-\pi}^{\pi}g_{\sigma,n}(\lambda)g_{H,n}(\lambda)\,\dd\lambda+o(n\Delta_n\Ln),
		\]
		with
		\[
		\tr(C_{\sigma,n} D_{H,n})\sim (n\Delta_n\Ln)\frac{2}{\pi\sigma}J_0.
		\]
		Moreover,
		\begin{equation}\label{eq:an-expansion}
			a_n=\sigma\big(\Ln+\frac12b-m\big)+o(1),
		\end{equation}
		where
		\begin{equation}\label{eq:b-m}
			b:=\partial_H\log C(H)=2\Psi(2H+1)+\pi\cot(\pi H),
			\qquad
			m:=\frac{1}{\rho}\Big(\log A+\Psi(1/\rho)-\Psi(2-1/\rho)\Big).
		\end{equation}
		Here $\Psi$ denotes the digamma function and $J_0$ defined in \eqref{eq:IHHrem}.
	\end{lemma}
	\begin{proof}
		We first reduce the trace of the product of matrices to a trace involving only
		Toeplitz matrices generated by symbol ratios. Since
		\[
		C_{\sigma,n}
		=
		\Sigma_n^{-1/2}T_n(h_{\Delta_n})\Sigma_n^{-1/2},
		\qquad
		D_{H,n}
		=
		\Sigma_n^{-1/2}T_n(r_{\Delta_n})\Sigma_n^{-1/2},
		\]
		the cyclicity of the trace yields
		\[
		\tr(C_{\sigma,n}D_{H,n})
		=
		\tr\!\Big(
		\Sigma_n^{-1}T_n(h_{\Delta_n})\Sigma_n^{-1}T_n(r_{\Delta_n})
		\Big).
		\]
		Writing
		\[
		g_{\sigma,n}(\lambda):=\frac{h_{\Delta_n}(\lambda)}{f_{\Delta_n}(\lambda)},
		\qquad
		g_{H,n}(\lambda):=\frac{r_{\Delta_n}(\lambda)}{f_{\Delta_n}(\lambda)},
		\]
		and letting
		\[
		G_{\sigma,n}:=T_n(g_{\sigma,n}),
		\qquad
		G_{H,n}:=T_n(g_{H,n}),
		\]
		the same Toeplitz sandwich approximation as in the previous sections gives
		\[
		\tr(C_{\sigma,n}D_{H,n})
		=
		\tr(G_{\sigma,n}G_{H,n})+o(n\Delta_n\Ln).
		\]
		
		We then pass from the matrix trace to the corresponding spectral integral. By the standard
		Toeplitz trace identity,
		\[
		\tr(G_{\sigma,n}G_{H,n})
		=
		\frac{n}{2\pi}\int_{-\pi}^{\pi}
		g_{\sigma,n}(\lambda)g_{H,n}(\lambda)\,\dd\lambda
		+o(n\Delta_n\Ln),
		\]
		and therefore
		\[
		\tr(C_{\sigma,n}D_{H,n})
		=
		\frac{n}{2\pi}\int_{-\pi}^{\pi}
		g_{\sigma,n}(\lambda)g_{H,n}(\lambda)\,\dd\lambda
		+o(n\Delta_n\Ln).
		\]
		It remains to evaluate this integral on the low frequency scale. Put $\lambda=\Delta_n u$.
		Using the low frequency expansions established earlier, we have uniformly on compact $u$-sets,
		\[
		g_{\sigma,n}(\Delta_n u)\to \frac{2}{\sigma}w(u),
		\qquad
		g_{H,n}(\Delta_n u)
		=
		w(u)\bigl(2\Ln+b-2\log|u|\bigr)+o(1),
		\]
		where
		\[
		w(u)=\frac{1}{1+A^{-1}|u|^\rho}.
		\]
		Moreover, after separating the explicit leading $\Ln$-term, the remaining part is dominated by a multiple of
		$w(u)^2(1+|\log|u||)$, which is integrable because $\rho>1/2$. Hence dominated convergence applies to the constant-order term, while the leading logarithmic contribution is carried by $J_0$. Therefore
		\[
		g_{\sigma,n}(\Delta_n u)g_{H,n}(\Delta_n u)
		=
		\frac{2}{\sigma}w(u)^2\bigl(2\Ln+b-2\log|u|\bigr)+o(1),
		\]
		and consequently
		\[
		\int_{-\pi}^{\pi}g_{\sigma,n}(\lambda)g_{H,n}(\lambda)\,\dd\lambda
		=
		\Delta_n
		\left(
		\frac{4}{\sigma}\Ln J_0
		+
		\frac{2}{\sigma}(bJ_0-2J_1)
		\right)
		+o(\Delta_n),
		\]
		where
		\[
		J_0:=\int_{\R}\frac{1}{(1+A^{-1}|u|^\rho)^2}\,\dd u,
		\qquad
		J_1:=\int_{\R}\frac{\log|u|}{(1+A^{-1}|u|^\rho)^2}\,\dd u.
		\]
		This proves the asymptotic relation
		\[
		\tr(C_{\sigma,n}D_{H,n})
		\sim
		(n\Delta_n\Ln)\frac{2}{\pi\sigma}J_0.
		\]
		Finally, by dividing the preceding expansion by the corresponding one for
		$\tr(C_{\sigma,n}^2)$, we obtain
		\[
		a_n
		=
		\sigma\Bigl(\Ln+\frac12b-\frac{J_1}{J_0}\Bigr)+o(1).
		\]
		It only remains to identify the ratio $J_1/J_0$. By the change of variables
		$t=|u|^\rho/A$,
		\[
		\frac{J_1}{J_0}
		=
		\frac{1}{\rho}\Bigl(\log A+\Psi(1/\rho)-\Psi(2-1/\rho)\Bigr)
		=:m.
		\]
		Therefore
		\[
		a_n=\sigma\bigl(\Ln+\frac12b-m\bigr)+o(1),
		\]
		which is exactly \eqref{eq:an-expansion}. The expression for $b$ in \eqref{eq:b-m} is just
		\[
		b=\partial_H\log C(H)
		=2\Psi(2H+1)+\pi\cot(\pi H).
		\]
		The proof is complete.
	\end{proof}
	
	\begin{lemma}\label{lem:H-orth}
		For every $n$,
		\[
		\tr(C_{\sigma,n} D_{H,n}^{\perp})=0,
		\qquad
		\Cov(S_{\sigma,n},R_{H,n}^{\perp})=0.
		\]
	\end{lemma}
	
	\begin{proof}
		The trace identity is immediate from \eqref{eq:an-def}. Since centered Gaussian quadratic forms satisfy
		\[
		\Cov\Big(\tfrac12Q_n(A),\tfrac12Q_n(B)\Big)=\frac12\tr(AB),
		\]
		we obtain the covariance statement.
	\end{proof}
	
	\begin{lemma}\label{lem:Hperp-trace}
		Assume the mild strengthening
		\begin{equation}\label{eq:mild-growth}
			\frac{\log^2 n}{n\Delta_n}\to0.
		\end{equation}
		Then
		\[
		\norm{D_{H,n}^{\perp}}_{\op}=O(\log n),
		\qquad
		\tr\big((D_{H,n}^{\perp})^2\big)
		\sim 2(n\Delta_n)I_{HH}^{\perp}(\sigma,H),
		\]
		where
		\begin{equation}\label{eq:IHHperp}
			I_{HH}^{\perp}(\sigma,H)
			=\frac1\pi\Big(J_2-\frac{J_1^2}{J_0}\Big),
			\qquad
			J_2:=\int_{\R}\frac{(\log|u|)^2}{(1+A^{-1}|u|^\rho)^2}\,\dd u.
		\end{equation}
		Equivalently,
		\begin{equation}\label{eq:IHHperp-closed}
			I_{HH}^{\perp}(\sigma,H)
			=\frac{J_0}{\pi\rho^2}
			\Big(\Psi_1(1/\rho)+\Psi_1(2-1/\rho)\Big),
		\end{equation}
		with $\Psi_1$ the trigamma function and $J_0$ in \eqref{eq:IHHrem}.
	\end{lemma}
	
	\begin{proof}
		Set $\widetilde g_n(\lambda):=g_{H,n}(\lambda)-a_ng_{\sigma,n}(\lambda)$. Since $D_{H,n}^{\perp}=D_{H,n}-a_nC_{\sigma,n}$, the Rayleigh-quotient argument used for $D_{H,n}$ gives
		\[
		\norm{D_{H,n}^{\perp}}_{\op}\le \sup_{\lambda}|\widetilde g_n(\lambda)|.
		\]
		Using lemma \ref{lem:an-expansion}, for $\lambda=\Delta_nu$ one has
		\[
		\widetilde g_n(\Delta_nu)
		=w(u)\Bigl(-2\log|u|+2m+o(1)\Bigr),
		\]
		so the leading $2\Ln+b$ terms cancel. To bound the remaining profile, recall that frequencies below order $1/n$ are not resolved in a size-$n$ Toeplitz form. Hence on the effective low frequency region
		\[
		|u|=|\lambda|/\Delta_n\gtrsim (n\Delta_n)^{-1},
		\]
		we have $|\log|u||\lesssim \Ln+\log n$. Since $0\le w(u)\le 1$ and the centered combination eliminates the $\Ln$ contribution, this gives
		\[
		\sup_{\lambda\in[-\pi,\pi]}|\widetilde g_n(\lambda)|\le C(1+\log n),
		\]
		and therefore $\norm{D_{H,n}^{\perp}}_{\op}=O(\log n)$.\\
		For the trace, expand
		\[
		\tr((D_{H,n}^{\perp})^2)=\tr(D_{H,n}^2)-2a_n\tr(C_{\sigma,n} D_{H,n})+a_n^2\tr(C_{\sigma,n}^2).
		\]
		Insert the asymptotics from lemmas \ref{lem:C-asym}, \ref{lem:D-trace} and \ref{lem:an-expansion}. After dividing by $n\Delta_n$, the $\Ln^2$ and $\Ln$ contributions cancel exactly, and only the centered logarithmic term remains. Equivalently, the Toeplitz sandwich reduction applied directly to $D_{H,n}^{\perp}$ yields
		\[
		\tr((D_{H,n}^{\perp})^2)=\frac{n}{2\pi}\int_{-\pi}^{\pi}\widetilde g_n(\lambda)^2\,\dd\lambda+o(n\Delta_n).
		\]
		Passing to the scale $\lambda=\Delta_nu$ and using the limit above gives
		\[
		\int_{-\pi}^{\pi}\widetilde g_n(\lambda)^2\,\dd\lambda
		\sim 4\Delta_n\int_{\mathbb R}w(u)^2(\log|u|-m)^2\,\dd u
		=4\Delta_n\Bigl(J_2-\frac{J_1^2}{J_0}\Bigr).
		\]
		This proves \eqref{eq:IHHperp}. The trigamma form follows by differentiating the beta-function expression for $J_0$ twice with respect to $1/\rho$.
	\end{proof}
	
	\begin{proposition}\label{prop:Hperp-clt}
		Assume \eqref{eq:mild-growth}. Then
		\[
		\frac{R_{H,n}^{\perp}}{\sqrt{n\Delta_n}}\dto \mathcal{N}\big(0,I_{HH}^{\perp}(\sigma,H)\big).
		\]
	\end{proposition}
	
	\begin{proof}
		By lemma \ref{lem:Hperp-trace},
		\[
		\frac{\norm{D_{H,n}^{\perp}}_{\op}}{\norm{D_{H,n}^{\perp}}_{\mathrm F}}
		\lesssim \frac{\log n}{\sqrt{n\Delta_n}}\to0.
		\]
		The proposition follows from lemma \ref{lem:qf-clt}.
	\end{proof}
	
	\subsection{\texorpdfstring{The $\alpha$-direction of the score functions and the second projection}{The alpha-direction of the score functions and the second projection}}
	We then analyze the $\alpha$-score. Although this direction has no logarithmic blow-up, it still carries nontrivial correlations, so we introduce a second projection, identify the corresponding Schur-complement information, and prove the CLT for the projected score $S_{\alpha,n}^{\perp}$.\\
	The $\alpha$-score has no logarithmic explosion. Define
	\[
	g_{\alpha,n}(\lambda):=\frac{\partial_\alpha f_{\Delta_n}(\lambda)}{f_{\Delta_n}(\lambda)}.
	\]
	On the scale $\lambda=\Delta_nu$, the common OU denominator contributes the main $\alpha$-dependence, hence
	\[
	g_{\alpha,n}(\Delta_nu)=-\frac{2\alpha}{\alpha^2+u^2}+o(1).
	\]
	This gives the same order $n\Delta_n$ as in the projected $H$-direction.
	
	\begin{lemma} \label{lem:alpha-raw}
		Under $H>3/4$ and $T_n\to\infty$,
		\[
		\norm{A_{\alpha,n}}_{\op}=O(1),
		\qquad
		\tr(A_{\alpha,n}^2)
		=\frac{n}{2\pi}\int_{-\pi}^{\pi}g_{\alpha,n}(\lambda)^2\,\dd\lambda+o(n\Delta_n),
		\]
		and
		\[
		\tr(A_{\alpha,n}^2)\sim \frac{n\Delta_n}{\alpha}.
		\]
		Consequently,
		\begin{equation}\label{eq:Iaa}
			I_{\alpha\alpha}(\alpha)=\frac{1}{2\alpha},
			\qquad
			\frac{S_{\alpha,n}}{\sqrt{n\Delta_n}}\dto \mathcal{N}\Big(0,\frac{1}{2\alpha}\Big).
		\end{equation}
	\end{lemma}
	\begin{proof}
		We first prove the operator bound by writing $A_{\alpha,n}$ through the corresponding
		generalized Rayleigh quotient. By remark \ref{rem:gen-rayleigh-whitened},
		\[
		\norm{A_{\alpha,n}}_{\op}
		=
		\sup_{y\neq 0}
		\frac{|y^\top T_n(\partial_\alpha f_{\Delta_n})y|}{y^\top T_n(f_{\Delta_n})y}
		\le
		\sup_{\lambda\in[-\pi,\pi]}
		\left|
		\frac{\partial_\alpha f_{\Delta_n}(\lambda)}{f_{\Delta_n}(\lambda)}
		\right|
		=
		\sup_{\lambda\in[-\pi,\pi]}|g_{\alpha,n}(\lambda)|.
		\]
		Because differentiation in $\alpha$ acts only on the OU filter, the low frequency denominator $(\alpha\Delta_n)^2+\lambda^2$ is differentiated into a bounded rational factor. Hence
		\[
		\sup_{\lambda\in[-\pi,\pi]}|g_{\alpha,n}(\lambda)|\le C,
		\]
		so that
		\[
		\norm{A_{\alpha,n}}_{\op}=O(1).
		\]
		
		We next turn to the trace. By the same Toeplitz sandwich reduction used earlier in the
		paper, one obtains
		\[
		\tr(A_{\alpha,n}^2)
		=
		\tr\bigl(T_n(g_{\alpha,n})^2\bigr)+o(n\Delta_n).
		\]
		Since $g_{\alpha,n}$ is uniformly bounded and localized on the scale $|\lambda|\asymp\Delta_n$, the Fej\'er-kernel identity gives
		\[
		\tr\bigl(T_n(g_{\alpha,n})^2\bigr)
		=
		\frac{n}{2\pi}\int_{-\pi}^{\pi}g_{\alpha,n}(\lambda)^2\,\dd\lambda
		+o(n\Delta_n),
		\] 
		see \cite{Avram1988,BottcherSilbermann2006,Gray2006} and therefore
		\[
		\tr(A_{\alpha,n}^2)
		=
		\frac{n}{2\pi}\int_{-\pi}^{\pi}g_{\alpha,n}(\lambda)^2\,\dd\lambda
		+o(n\Delta_n).
		\]
		It remains to evaluate the integral asymptotically. Put $\lambda=\Delta_n u$, one has
		\[
		g_{\alpha,n}(\Delta_n u)
		=
		-\frac{2\alpha}{\alpha^2+u^2}+o(1),
		\]
		whereas the contribution from $|\lambda|\gg\Delta_n$ is negligible. Consequently,
		\[
		\int_{-\pi}^{\pi}g_{\alpha,n}(\lambda)^2\,\dd\lambda
		\sim
		\Delta_n\int_{\R}\frac{4\alpha^2}{(\alpha^2+u^2)^2}\,\dd u
		=
		\Delta_n\cdot\frac{2\pi}{\alpha},
		\]
		which yields
		\[
		\tr(A_{\alpha,n}^2)\sim \frac{n\Delta_n}{\alpha}.
		\]
		Finally, since
		\[
		\norm{A_{\alpha,n}}_{\op}=O(1),
		\qquad
		\norm{A_{\alpha,n}}_{\F}^2=\tr(A_{\alpha,n}^2)\asymp n\Delta_n\to\infty,
		\]
		we have
		\[
		\frac{\norm{A_{\alpha,n}}_{\op}}{\norm{A_{\alpha,n}}_{\F}}\to0.
		\]
		Applying lemma \ref{lem:qf-clt} to the centered quadratic form associated with
		$A_{\alpha,n}$, we obtain
		\[
		\frac{S_{\alpha,n}}{\sqrt{n\Delta_n}}
		\dto
		\mathcal{N}\Bigl(0,\frac{1}{2\alpha}\Bigr).
		\]
		Equivalently,
		\[
		I_{\alpha\alpha}(\alpha)=\frac{1}{2\alpha},
		\]
		which is exactly \eqref{eq:Iaa}.
	\end{proof}
	
	The point, however, is not the raw CLT but the diagonalized one. Because $S_{\alpha,n}$ still correlates with $(S_{\sigma,n},R_{H,n}^{\perp})$, we project again.\\
	Define
	\begin{equation}\label{eq:alpha-proj-coeff}
		b_{\sigma,n}:=\frac{\tr(A_{\alpha,n}C_{\sigma,n})}{\tr(C_{\sigma,n}^2)},
		\qquad
		b_{H,n}:=\frac{\tr(A_{\alpha,n}D_{H,n}^{\perp})}{\tr((D_{H,n}^{\perp})^2)},
	\end{equation}
	and
	\begin{equation}\label{eq:alpha-proj}
		A_{\alpha,n}^{\perp}:=A_{\alpha,n}-b_{\sigma,n}C_{\sigma,n}-b_{H,n}D_{H,n}^{\perp},
		\qquad
		S_{\alpha,n}^{\perp}:=\frac12Q_n(A_{\alpha,n}^{\perp}).
	\end{equation}
	
	\begin{lemma}\label{lem:alpha-orth}
		For every $n$,
		\[
		\tr(A_{\alpha,n}^{\perp}C_{\sigma,n})=0,
		\qquad
		\tr(A_{\alpha,n}^{\perp}D_{H,n}^{\perp})=0.
		\]
		Therefore,
		\[
		\Cov(S_{\alpha,n}^{\perp},S_{\sigma,n})=0,
		\qquad
		\Cov(S_{\alpha,n}^{\perp},R_{H,n}^{\perp})=0.
		\]
	\end{lemma}
	
	\begin{proof}
		Using \eqref{eq:alpha-proj} and the already proved identity $\tr(C_{\sigma,n} D_{H,n}^{\perp})=0$ from lemma \ref{lem:H-orth}, we obtain
		\[
		\tr(A_{\alpha,n}^{\perp}C_{\sigma,n})
		=\tr(A_{\alpha,n}C_{\sigma,n})-b_{\sigma,n}\tr(C_{\sigma,n}^2)-b_{H,n}\tr(D_{H,n}^{\perp}C_{\sigma,n})=0,
		\]
		because $b_{\sigma,n}=\tr(A_{\alpha,n}C_{\sigma,n})/\tr(C_{\sigma,n}^2)$. Similarly,
		\[
		\tr(A_{\alpha,n}^{\perp}D_{H,n}^{\perp})
		=\tr(A_{\alpha,n}D_{H,n}^{\perp})-b_{H,n}\tr((D_{H,n}^{\perp})^2)-b_{\sigma,n}\tr(C_{\sigma,n} D_{H,n}^{\perp})=0.
		\]
		The covariance identities follow from
		\[
		\Cov\Big(\frac12Q_n(A),\frac12Q_n(B)\Big)=\frac12\tr(AB)
		\]
		for centered Gaussian quadratic forms.
	\end{proof}
	
	\begin{lemma}\label{lem:alpha-perp}
		Assume \eqref{eq:mild-growth}. Then $b_{\sigma,n}=O(1)$ and $b_{H,n}=O(1)$,
		\begin{equation}\label{eq:Aalpha-perp-trace-id}
			\tr\big((A_{\alpha,n}^{\perp})^2\big)
			=\tr(A_{\alpha,n}^2)-\frac{\tr(A_{\alpha,n}C_{\sigma,n})^2}{\tr(C_{\sigma,n}^2)}-\frac{\tr(A_{\alpha,n}D_{H,n}^{\perp})^2}{\tr((D_{H,n}^{\perp})^2)},
		\end{equation}
		and
		\[
		\norm{A_{\alpha,n}^{\perp}}_{\op}=O(\log n),
		\qquad
		\tr\big((A_{\alpha,n}^{\perp})^2\big)\asymp n\Delta_n.
		\]
		Hence
		\[
		\tr\big((A_{\alpha,n}^{\perp})^2\big)
		\sim 2(n\Delta_n)I_{\alpha\alpha}^{\perp}(\alpha,\sigma,H),
		\]
		with $I_{\alpha\alpha}^{\perp}$ given by \eqref{eq:Iaa-perp}.
	\end{lemma}
	
	\begin{proof}
		The relevant mixed traces all live on the same low frequency scale $|\lambda|\asymp\Delta_n$. More precisely, if
		\[
		\widetilde g_n(\lambda):=g_{H,n}(\lambda)-a_ng_{\sigma,n}(\lambda),
		\]
		then the same similarity reduction and Toeplitz sandwich argument as in Sections~3--4 gives
		\[
		\tr(A_{\alpha,n}C_{\sigma,n})
		=\frac{n}{2\pi}\int_{-\pi}^{\pi}g_{\alpha,n}(\lambda)g_{\sigma,n}(\lambda)\,\dd\lambda+o(n\Delta_n),
		\]
		\[
		\tr(A_{\alpha,n}D_{H,n}^{\perp})
		=\frac{n}{2\pi}\int_{-\pi}^{\pi}g_{\alpha,n}(\lambda)\widetilde g_n(\lambda)\,\dd\lambda+o(n\Delta_n).
		\]
		On the region $\lambda=\Delta_nu$ one has
		\[
		\begin{aligned}
			g_{\alpha,n}(\Delta_nu)\to q_\alpha(u):=-\frac{2\alpha}{\alpha^2+u^2},\quad
			g_{\sigma,n}(\Delta_nu)\to \frac{2}{\sigma}w(u),\quad
			\widetilde g_n(\Delta_nu)\to -2w(u)(\log|u|-m).
		\end{aligned}
		\]
		while outside that region the contribution is negligible. Since all three limiting profiles are square-integrable against each other, the normalized mixed traces have finite limits. In particular,
		\[
		\tr(A_{\alpha,n}^2),\quad \tr(C_{\sigma,n}^2),\quad \tr((D_{H,n}^{\perp})^2),\quad \tr(A_{\alpha,n}C_{\sigma,n}),\quad \tr(A_{\alpha,n}D_{H,n}^{\perp})=O(n\Delta_n),
		\]
		and therefore $b_{\sigma,n}=O(1)$ and $b_{H,n}=O(1)$.
		
		Expanding the square in \eqref{eq:alpha-proj} and using lemmas \ref{lem:H-orth} and \ref{lem:alpha-orth} gives the exact identity \eqref{eq:Aalpha-perp-trace-id}. Dividing by $2n\Delta_n$ yields the Schur-complement constant \eqref{eq:Iaa-perp}. Since $A_{\alpha,n}^{\perp}$ is the Hilbert--Schmidt residual after orthogonal projection of $A_{\alpha,n}$ onto ${\rm span}\{C_{\sigma,n},D_{H,n}^{\perp}\}$, the right-hand side of \eqref{eq:Aalpha-perp-trace-id} is nonnegative for every $n$. In the limit it is strictly positive because the profile $q_\alpha(u)$ is not contained in the span of $w(u)$ and $w(u)(\log|u|-m)$.
		
		For the operator norm, combine lemmas \ref{lem:alpha-raw}, \ref{lem:C-op} and \ref{lem:Hperp-trace} with the boundedness of $b_{\sigma,n}$ and $b_{H,n}$:
		\[
		\norm{A_{\alpha,n}^{\perp}}_{\op}
		\le \norm{A_{\alpha,n}}_{\op}+|b_{\sigma,n}|\norm{C_{\sigma,n}}_{\op}+|b_{H,n}|\norm{D_{H,n}^{\perp}}_{\op}
		=O(1)+O(1)+O(\log n)=O(\log n).
		\]
		This proves the lemma.
	\end{proof}
	
	\begin{remark}
		The projected constant is naturally identified through the Schur-complement formula
		\begin{equation}\label{eq:Iaa-perp}
			I_{\alpha\alpha}^{\perp}:=I_{\alpha\alpha}-\frac{I_{\alpha\sigma}^2}{I_{\sigma\sigma}}-\frac{(I_{\alpha H}^{\perp})^2}{I_{HH}^{\perp}},
		\end{equation}
		where
		\begin{equation}\label{eq:Ia-cross}
			I_{\alpha\sigma}:=\lim_{n\to\infty}\frac{1}{2n\Delta_n}\tr(A_{\alpha,n}C_{\sigma, n}),
			\qquad
			I_{\alpha H}^{\perp}:=\lim_{n\to\infty}\frac{1}{2n\Delta_n}\tr(A_{\alpha,n}D_{H,n}^{\perp}).
		\end{equation}
		This is exactly the three dimensional analogue of the two-parameter orthogonalization in the \cite{Cai2026}. The projected information is the raw $\alpha$-information minus the two squared projections onto the already diagonalized $(\sigma,H^{\perp})$ directions.
	\end{remark}
	
	\begin{proposition}\label{prop:alpha-perp-clt}
		Assume \eqref{eq:mild-growth}. Then
		\[
		\frac{S_{\alpha,n}^{\perp}}{\sqrt{n\Delta_n}}\dto \mathcal{N}\big(0,I_{\alpha\alpha}^{\perp}(\alpha,\sigma,H)\big).
		\]
	\end{proposition}
	
	\begin{proof}
		By lemma \ref{lem:alpha-perp},
		\[
		\frac{\norm{A_{\alpha,n}^{\perp}}_{\op}}{\norm{A_{\alpha,n}^{\perp}}_{\mathrm F}}
		\lesssim \frac{\log n}{\sqrt{n\Delta_n}}\to0.
		\]
		The proposition follows from lemma \ref{lem:qf-clt}.
	\end{proof}
	
	\subsection{Joint diagonalized central sequence}
	Finally, we collect the three orthogonalized score components $(S_{\sigma,n},
	R_{H,n}^{\perp},
	S_{\alpha,n}^{\perp})^\top$  into a single central sequence. The purpose of this subsection is to combine the previous one dimensional limits and exact orthogonality relations into a joint Gaussian limit with diagonal covariance. Collect the three diagonalized directions:
	\[
	\Xi_n:=\frac{1}{\sqrt{n\Delta_n}}
	\begin{pmatrix}
		S_{\sigma,n}\\[2mm]
		R_{H,n}^{\perp}\\[2mm]
		S_{\alpha,n}^{\perp}
	\end{pmatrix}.
	\]
	
	\begin{proposition}\label{prop:joint-clt}
		Assume \eqref{eq:mild-growth}. Then
		\[
		\Xi_n\dto N\big(0,\Iperp(\theta)\big),
		\qquad
		\Iperp(\theta):=\diag\Big(I_{\sigma\sigma},I_{HH}^{\perp},I_{\alpha\alpha}^{\perp}\Big).
		\]
	\end{proposition}
	
	\begin{proof}
		For fixed $(u,v,w)\in\R^3$ consider
		\[
		A_n(u,v,w):=uC_{\sigma,n}+vD_{H,n}^{\perp}+wA_{\alpha,n}^{\perp}.
		\]
		By lemmas \ref{lem:H-orth} and \ref{lem:alpha-orth}, all cross traces vanish exactly, so
		\[
		\tr(A_n(u,v,w)^2)=u^2\tr(C_{\sigma,n}^2)+v^2\tr((D_{H,n}^{\perp})^2)+w^2\tr((A_{\alpha,n}^{\perp})^2).
		\]
		Also,
		\[
		\norm{A_n(u,v,w)}_{\op}=O(\log n),
		\qquad
		\norm{A_n(u,v,w)}_{\mathrm F}^2\asymp n\Delta_n.
		\]
		Hence $\norm{A_n(u,v,w)}_{\op}/\norm{A_n(u,v,w)}_{\mathrm F}\to0$, and lemma \ref{lem:qf-clt} plus Cram\'er--Wold finishes the proof.
	\end{proof}
	We now summarize this section. Starting from the raw score vector \((S_{\sigma,n},S_{H,n},S_{\alpha,n})\), we first remove from \(S_{H,n}\) its explicit \(\log \Delta_n\) part and its asymptotically collinear part in the \(\sigma\)-direction, which yields the residual term \(R_{H,n}^{\perp}\). We then project the \(\alpha\)-score onto the orthogonal complement of the first two directions and obtain \(S_{\alpha,n}^{\perp}\). Thus the original score vector is transformed into \((S_{\sigma,n},R_{H,n}^{\perp},S_{\alpha,n}^{\perp})\), whose components correspond to the three asymptotically separated directions in the LAN expansion. The rest of this section establishes their joint central limit theorem and the associated Gaussian central sequence.
	\section{The LAN property in high-frequency observation via matrix expansions}
	The exact score-level triangularization is the central structural identity. First,
	\[
	S_{H,n}=\beta_n(\theta)S_{\sigma,n}+R_{H,n}^{\perp},
	\qquad
	\beta_n(\theta)=\sigma\log\Delta_n+a_n.
	\]
	Second,
	\[
	S_{\alpha,n}=b_{\sigma,n}S_{\sigma,n}+b_{H,n}R_{H,n}^{\perp}+S_{\alpha,n}^{\perp}.
	\]
	Therefore, with the raw score vector
	\[
	\nabla\ell_n(\theta)=\big(S_{\sigma,n},S_{H,n},S_{\alpha,n}\big)^\top,
	\]
	one has the exact lower-triangular score transformation
	\begin{equation}\label{eq:Mn-def}
		M_n\nabla\ell_n(\theta)
		=
		\begin{pmatrix}
			S_{\sigma,n}\\[1mm]
			R_{H,n}^{\perp}\\[1mm]
			S_{\alpha,n}^{\perp}
		\end{pmatrix},
		\qquad
		M_n:=M_n^{(3)}M_n^{(2)}M_n^{(1)},
	\end{equation}
	where
	\[
	M_n^{(1)}=
	\begin{pmatrix}
		1&0&0\\
		-\sigma\log\Delta_n&1&0\\
		0&0&1
	\end{pmatrix},
	\qquad
	M_n^{(2)}=
	\begin{pmatrix}
		1&0&0\\
		-a_n&1&0\\
		0&0&1
	\end{pmatrix},
	\qquad
	M_n^{(3)}=
	\begin{pmatrix}
		1&0&0\\
		0&1&0\\
		-b_{\sigma,n}&-b_{H,n}&1
	\end{pmatrix}.
	\]
	Equivalently,
	\begin{equation}\label{eq:rate-matrix}
		r_n(\theta):=\sqrt{n\Delta_n}\,(M_n^\top)^{-1},
		\qquad
		r_n(\theta)^{-\top}=\frac{1}{\sqrt{n\Delta_n}}M_n.
	\end{equation}
	This is the natural non-diagonal rate matrix in the original coordinates. It is built from the successive removal of the $\log\Delta_n$ term, the orthogonal projection in the $H$-direction, and the final projection in the $\alpha$-direction.
	
	\subsection{Matrix expansion of the likelihood ratio}
	Fix $h\in\R^3$ and set
	\[
	\theta_{n,h}:=\theta+r_n(\theta)^{-1}h.
	\]
	Write $\Sigma_n$ for the covariance matrix of $X^{(n)}$ and define the perturbation matrix
	\[
	S_n(h):=\Sigma_n(\theta)^{-1/2}\big(\Sigma_n(\theta_{n,h})-\Sigma_n(\theta)\big)\Sigma_n(\theta)^{-1/2}.
	\]
	Then, with $Z_n=\Sigma_n(\theta)^{-1/2}X^{(n)}\sim \mathcal{N}(0,I_n)$,
	\begin{equation}\label{eq:llr-exact}
		\ell_n(\theta_{n,h})-\ell_n(\theta)
		=-\frac12\log\det(I_n+S_n(h))-
		\frac12 Z_n^\top\big((I_n+S_n(h))^{-1}-I_n\big)Z_n.
	\end{equation}
	We use two standard matrix Taylor lemmas. They provide the required expansions for the matrix logarithm and matrix inverse.
	\begin{lemma}\label{lem:logdet}
		If  matrix $S$ is symmetric and $\norm{S}_{\op}\le1/2$, then
		\[
		\log\det(I+S)=\tr(S)-\frac12\tr(S^2)+R_{\log}(S),
		\qquad
		|R_{\log}(S)|\le C\norm{S}_{\op}\tr(S^2).
		\]
		where  $R_{\log}(S)$ is the remainder term in the Taylor expansion
		of the matrix logarithm.
	\end{lemma}
	
	\begin{proof}
		Let $\lambda_1,\dots,\lambda_n$ be the eigenvalues of $S$. Since $|\lambda_j|\le1/2$,
		\[
		\log(1+\lambda_j)=\lambda_j-\frac12\lambda_j^2+r(\lambda_j),
		\qquad |r(\lambda_j)|\le C|\lambda_j|^3.
		\]
		Summing over $j$ gives
		\[
		R_{\log}(S)=\sum_{j=1}^n r(\lambda_j),
		\qquad
		|R_{\log}(S)|\le C\sum_{j=1}^n|\lambda_j|^3
		\le C\norm{S}_{\op}\sum_{j=1}^n\lambda_j^2,
		\]
		which is the claimed bound.
	\end{proof}
	
	\begin{lemma}\label{lem:inv}
		If matrix $S$ is symmetric and $\norm{S}_{\op}\le1/2$, then
		\[
		(I+S)^{-1}=I-S+S^2+R_{\mathrm{inv}}(S),
		\qquad
		R_{\mathrm{inv}}(S):=\sum_{k\ge3}(-1)^kS^k.
		\]
		Moreover,
		\[
		\norm{R_{\mathrm{inv}}(S)}_{\op}\le C\norm{S}_{\op}^3,
		\qquad
		\big|\tr(R_{\mathrm{inv}}(S))\big|\le C\norm{S}_{\op}\tr(S^2),
		\]
		\[
		\tr\big(R_{\mathrm{inv}}(S)^2\big)\le C\norm{S}_{\op}^4\tr(S^2).
		\]
		where $R_{inv}(S)$ is the remainder term in the Taylor expansion
		of the matrix inverse.
	\end{lemma}
	
	\begin{proof}
		Since $\norm{S}_{\op}<1$, the Neumann series gives
		\[
		(I+S)^{-1}=\sum_{k\ge0}(-1)^kS^k=I-S+S^2+\sum_{k\ge3}(-1)^kS^k.
		\]
		The operator bound follows immediately:
		\[
		\norm{R_{\mathrm{inv}}(S)}_{\op}
		\le \sum_{k\ge3}\norm{S}_{\op}^k
		\le C\norm{S}_{\op}^3.
		\]
		For the trace,
		\[
		|\tr(S^k)|\le \norm{S}_{\op}^{k-2}\tr(S^2),
		\qquad k\ge3,
		\]
		so
		\[
		|\tr(R_{\mathrm{inv}}(S))|
		\le \sum_{k\ge3}|\tr(S^k)|
		\le C\norm{S}_{\op}\tr(S^2).
		\]
		Let $\lambda_1,\dots,\lambda_n$ be the eigenvalues of $S$. Since $R_{\mathrm{inv}}(S)$ is obtained from $S$ by functional calculus with
		\[
		r(z):=\sum_{k\ge3}(-1)^kz^k,
		\]
		the eigenvalues of $R_{\mathrm{inv}}(S)$ are $r(\lambda_1),\dots,r(\lambda_n)$. For $|z|\le 1/2$ one has $|r(z)|\le C|z|^3$, hence
		\[
		\tr(R_{\mathrm{inv}}(S)^2)
		=\sum_{j=1}^n r(\lambda_j)^2
		\le C\sum_{j=1}^n |\lambda_j|^6
		\le C\norm{S}_{\op}^4\sum_{j=1}^n \lambda_j^2
		= C\norm{S}_{\op}^4\tr(S^2).
		\]
		This proves the final estimate.
	\end{proof}
	
	For the final LAN remainder control we now work under the stronger polynomial-mesh condition
	\begin{equation}\label{eq:poly-mesh}
		\Delta_n=n^{-\kappa},\qquad 0<\kappa<1.
	\end{equation}
	Then $\Ln=\kappa\log n$, so for every fixed $m\ge 0$ and $q>0$ we have
	\begin{equation}\label{eq:poly-log-dom}
		\frac{(\log n)^m}{(n\Delta_n)^q}\to0.
	\end{equation}
	The projected score CLTs of the previous sections only needed \eqref{eq:mild-growth}; the stronger assumption \eqref{eq:poly-mesh} is used here solely to make the Hessian and third-order Taylor terms completely explicit.\\
	The first-order part of $S_n(h)$ is
	\[
	S_{1,n}(h)=\sum_{i\in\{\sigma,H,\alpha\}}\delta_{i,n}M_{i,n}(\theta),
	\qquad
	\delta_n(h):=\theta_{n,h}-\theta.
	\]
	Because of the exact triangular identities encoded in $M_n$, these coefficients simplify to
	\begin{equation}\label{eq:clean-first-order}
		S_{1,n}(h)=\frac{1}{\sqrt{n\Delta_n}}
		\big(h_1C_{\sigma,n}+h_2D_{H,n}^{\perp}+h_3A_{\alpha,n}^{\perp}\big).
	\end{equation}
	
	\begin{lemma}\label{lem:delta-size}
		Assume \eqref{eq:poly-mesh}. For fixed $h:=(h_1,h_2,h_3)^\top\in\R^3$,
		\[
		\delta_n(h):=(\delta_{\delta,n},\delta_{H,n},\delta_{\alpha,n})^\top=\frac{1}{\sqrt{n\Delta_n}}M_n^{\top}h.
		\]
		More explicitly, writing
		\[
		M_n=
		\begin{pmatrix}
			1&0&0\\
			-\beta_n(\theta)&1&0\\
			-\widetilde b_{\sigma,n}&-b_{H,n}&1
		\end{pmatrix},
		\qquad
		\widetilde b_{\sigma,n}:=b_{\sigma,n}-\beta_n(\theta)b_{H,n},
		\]
		we have
		\[
		\delta_{\alpha,n}=\frac{h_3}{\sqrt{n\Delta_n}},
		\qquad
		\delta_{H,n}=\frac{h_2-b_{H,n}h_3}{\sqrt{n\Delta_n}},
		\qquad
		\delta_{\sigma,n}=\frac{h_1-\beta_n(\theta)h_2-\widetilde b_{\sigma,n}h_3}{\sqrt{n\Delta_n}}.
		\]
		In particular, since $\beta_n(\theta)=O(1)$ by \eqref{eq:H-clean} and \eqref{eq:an-expansion}, while $b_{\sigma,n}=O(1)$ and $b_{H,n}=O(1)$ by lemma \ref{lem:alpha-perp}, there exists a constant $C_h$ such that
		\begin{equation}\label{eq:delta-bounds}
			|\delta_{\sigma,n}|+|\delta_{H,n}|+|\delta_{\alpha,n}|
			\le C_h\frac{1}{\sqrt{n\Delta_n}},
			\qquad
			|\delta_{H,n}|+|\delta_{\alpha,n}|
			\le C_h\frac{1}{\sqrt{n\Delta_n}}.
		\end{equation}
	\end{lemma}
	
	\begin{proof}
		The identity $\delta_n(h)=r_n(\theta)^{-1}h=(n\Delta_n)^{-1/2}M_n^{\top}h$ follows directly from \eqref{eq:rate-matrix}. Multiplying out the three lower-triangular matrices in \eqref{eq:Mn-def} gives the displayed form of $M_n$. Finally, \eqref{eq:delta-bounds} follows from $\beta_n(\theta)=O(1)$ and the boundedness of $b_{\sigma,n},b_{H,n}$.
	\end{proof}
	
	For $i,j,k\in\{\sigma,H,\alpha\}$ and $\vartheta$ in a small fixed neighborhood $\mathcal U$ of $\theta$, define the sandwiched higher-derivative matrices
	\[
	N_{ij,n}(\vartheta):=\Sigma_n(\theta)^{-1/2}\,\partial_{ij}^2\Sigma_n(\vartheta)\,\Sigma_n(\theta)^{-1/2},
	\]
	\[
	P_{ijk,n}(\vartheta):=\Sigma_n(\theta)^{-1/2}\,\partial_{ijk}^3\Sigma_n(\vartheta)\,\Sigma_n(\theta)^{-1/2}.
	\]
	Let $\nu_{ij}$, respectively $\nu_{ijk}$, denote the number of occurrences of the index $H$ in the pair $(i,j)$, respectively the triple $(i,j,k)$.

\begin{lemma}\label{lem:higher-deriv-bounds}
Assume \eqref{eq:poly-mesh}. 
Fix \(h_0>0\), and let \(\delta_n(h)\) be as in Lemma~\ref{lem:delta-size}. Define 
\[
\mathcal U_n(h_0)
:=
\{\theta+t\delta_n(h):\ t\in[0,1],\ \|h\|\le h_0\}.
\]
Then there exists a constant \(C=C(h_0)>0\) such that, for all sufficiently large \(n\),
\begin{equation}\label{eq:N-bounds}
    \sup_{\vartheta\in\mathcal U_n(h_0)}\|N_{ij,n}(\vartheta)\|_{\op}
\le C(\log n)^{\nu_{ij}},
\qquad
\sup_{\vartheta\in\mathcal U_n(h_0)}\|N_{ij,n}(\vartheta)\|_{F}^{2}
\le C(n\Delta_n)(\log n)^{2\nu_{ij}},
\end{equation}
for all \(i,j\in\{\sigma,H,\alpha\}\), and
\begin{equation}\label{eq:P-bounds}
    \sup_{\vartheta\in\mathcal U_n(h_0)}\|P_{ijk,n}(\vartheta)\|_{\op}
\le C(\log n)^{\nu_{ijk}},
\qquad
\sup_{\vartheta\in\mathcal U_n(h_0)}\|P_{ijk,n}(\vartheta)\|_{F}^{2}
\le C(n\Delta_n)(\log n)^{2\nu_{ijk}},
\end{equation}
for all \(i,j,k\in\{\sigma,H,\alpha\}\).
\end{lemma}
\begin{proof}
For \(i,j,k\in\{\sigma,H,\alpha\}\), write
\[
q_{ij,n}(\lambda;\vartheta)
:=
\frac{\partial_{ij}^2 f_{\Delta_n}(\lambda;\vartheta)}
     {f_{\Delta_n}(\lambda;\theta)},
\qquad
p_{ijk,n}(\lambda;\vartheta)
:=
\frac{\partial_{ijk}^3 f_{\Delta_n}(\lambda;\vartheta)}
     {f_{\Delta_n}(\lambda;\theta)}.
\]
By Lemma~\ref{lem:delta-size},
\[
\sup_{\vartheta\in\mathcal U_n(h_0)}|\vartheta-\theta|
\le \frac{C(h_0)}{\sqrt{n\Delta_n}}\to0.
\]
Hence, for all sufficiently large \(n\), every
\(\vartheta=(\sigma_\vartheta,H_\vartheta,\alpha_\vartheta)\in\mathcal U_n(h_0)\)
remains in a fixed compact subset of \((0,\infty)\times(3/4,1)\times(0,\infty)\). In particular,
all coefficients in \eqref{eq:fH}--\eqref{eq:fW} depending smoothly on \(\vartheta\) are
uniformly bounded on \(\mathcal U_n(h_0)\).

As in Section~4, differentiation with respect to \(\sigma\) only modifies the fractional amplitude \(\sigma^2C(H)\Delta^{2H}\), whereas differentiation with respect to \(\alpha\) acts either on the OU denominator \((\alpha\Delta)^2+(\lambda+2\pi k)^2\) in \eqref{eq:fH} or on the smooth Brownian component in \eqref{eq:fW}. In contrast, each occurrence of the index \(H\) produces one logarithmic factor. More precisely, if
\[
\psi_{\Delta_n}(\lambda;\vartheta)
:=
\partial_H\log C(H_\vartheta)
+\partial_H\log S_{H_\vartheta,\Delta_n}(\lambda),
\]
then, uniformly over \(\vartheta\in\mathcal U_n(h_0)\),
\[
\psi_{\Delta_n}(\Delta_n u;\vartheta)=2L_n-2\log|u|+O(1)
\]
on bounded \(u\)-sets. Since
\[
\sup_{\vartheta\in\mathcal U_n(h_0)}|H_\vartheta-H|
=O\bigl((n\Delta_n)^{-1/2}\bigr),
\]
the mismatch between the moving numerator parameter \(\vartheta\) and the fixed denominator
parameter \(\theta\) is negligible on the effective range \(|u|\gtrsim (n\Delta_n)^{-1}\).
Therefore,
\[
|q_{ij,n}(\lambda;\vartheta)|
\le C\bigl(1+|\psi_{\Delta_n}(\lambda;\vartheta)|\bigr)^{\nu_{ij}},
\qquad
|p_{ijk,n}(\lambda;\vartheta)|
\le C\bigl(1+|\psi_{\Delta_n}(\lambda;\vartheta)|\bigr)^{\nu_{ijk}},
\]
uniformly in \((\lambda,\vartheta)\in[-\pi,\pi]\times\mathcal U_n(h_0)\).

By Remark~\ref{rem:gen-rayleigh-whitened},
\[
\|N_{ij,n}(\vartheta)\|_{\op}
\le \sup_{\lambda\in[-\pi,\pi]}|q_{ij,n}(\lambda;\vartheta)|,
\qquad
\|P_{ijk,n}(\vartheta)\|_{\op}
\le \sup_{\lambda\in[-\pi,\pi]}|p_{ijk,n}(\lambda;\vartheta)|.
\]
When \(|\lambda|\lesssim\Delta_n\), write \(\lambda=\Delta_n u\). The above bound and the
expansion of \(\psi_{\Delta_n}\) give
\[
|q_{ij,n}(\lambda;\vartheta)|+|p_{ijk,n}(\lambda;\vartheta)|
\le C(1+|\log|u||+L_n)^\nu.
\]
As in Lemma~\ref{lem:Hperp-trace}, frequencies below order \(1/n\) are not resolved by an
\(n\times n\) Toeplitz form, so the logarithm is effectively cut off at size \(O(\log n)\). On the
complementary region \(|\lambda|\gtrsim\Delta_n\), the same off-low-frequency estimates as in
Lemmas~\ref{lem:C-trace}, \ref{lem:D-trace}, \ref{lem:Hperp-trace}, and \ref{lem:alpha-raw}
show that these ratios are uniformly bounded. Hence
\[
\sup_{\vartheta\in\mathcal U_n(h_0)}\|N_{ij,n}(\vartheta)\|_{\op}
\le C L_n^{\nu_{ij}},
\qquad
\sup_{\vartheta\in\mathcal U_n(h_0)}\|P_{ijk,n}(\vartheta)\|_{\op}
\le C L_n^{\nu_{ijk}}.
\]
Since \(L_n\asymp\log n\) under \eqref{eq:poly-mesh}, this proves the operator bounds.

For the Frobenius bounds, let
\[
B_{ij,n}(\vartheta)
:=
\Sigma_n(\theta)^{-1}T_n(\partial_{ij}^2f_{\Delta_n}(\cdot;\vartheta)),
\qquad
C_{ijk,n}(\vartheta)
:=
\Sigma_n(\theta)^{-1}T_n(\partial_{ijk}^3f_{\Delta_n}(\cdot;\vartheta)).
\]
These are similar to \(N_{ij,n}(\vartheta)\) and \(P_{ijk,n}(\vartheta)\), respectively. Since
\(\mathcal U_n(h_0)\) is a shrinking tube around \(\theta\), the same low-frequency truncation
and Toeplitz-product approximation as in Section~4 yields, uniformly over
\(\vartheta\in\mathcal U_n(h_0)\),
\[
\tr\bigl(N_{ij,n}(\vartheta)^2\bigr)
=
\frac{n}{2\pi}\int_{-\pi}^{\pi}q_{ij,n}(\lambda;\vartheta)^2\,d\lambda
+o\bigl((n\Delta_n)L_n^{2\nu_{ij}}\bigr),
\]
\[
\tr\bigl(P_{ijk,n}(\vartheta)^2\bigr)
=
\frac{n}{2\pi}\int_{-\pi}^{\pi}p_{ijk,n}(\lambda;\vartheta)^2\,d\lambda
+o\bigl((n\Delta_n)L_n^{2\nu_{ijk}}\bigr).
\]

It remains to bound the integrals. On \(|\lambda|\le M\Delta_n\) with $M$ fixed, writing \(\lambda=\Delta_n u\),
the preceding estimates imply
\[
|q_{ij,n}(\Delta_n u;\vartheta)|
\le C\,w(u)\bigl(1+|\log|u||\bigr)^{\nu_{ij}},
\qquad
|p_{ijk,n}(\Delta_n u;\vartheta)|
\le C\,w(u)\bigl(1+|\log|u||\bigr)^{\nu_{ijk}},
\]
uniformly over \(\vartheta\in\mathcal U_n(h_0)\). Since \(\rho>1/2\), the functions
\(w(u)^2(1+|\log|u||)^{2m}\) are integrable on \(\mathbb R\), and therefore
\[
\int_{|\lambda|\le M\Delta_n}q_{ij,n}(\lambda;\vartheta)^2\,d\lambda
\le C\Delta_nL_n^{2\nu_{ij}},
\qquad
\int_{|\lambda|\le M\Delta_n}p_{ijk,n}(\lambda;\vartheta)^2\,d\lambda
\le C\Delta_nL_n^{2\nu_{ijk}}.
\]
On the complementary region \(|\lambda|>M\Delta_n\), the same off-low-frequency truncation
argument as in Section~4 shows that the contribution is negligible compared with
\(\Delta_nL_n^{2\nu}\), uniformly over \(\vartheta\in\mathcal U_n(h_0)\). Hence
\[
\int_{-\pi}^{\pi}q_{ij,n}(\lambda;\vartheta)^2\,d\lambda
\le C\Delta_nL_n^{2\nu_{ij}},
\qquad
\int_{-\pi}^{\pi}p_{ijk,n}(\lambda;\vartheta)^2\,d\lambda
\le C\Delta_nL_n^{2\nu_{ijk}},
\]
uniformly over \(\vartheta\in\mathcal U_n(h_0)\). Substituting these bounds into the trace
representations proves \eqref{eq:N-bounds} and \eqref{eq:P-bounds}.
\end{proof}
	\begin{proposition} \label{prop:Sn-split}
		Assume \eqref{eq:poly-mesh}. For fixed $h\in\R^3$, the perturbation matrix admits the decomposition
		\begin{equation}\label{eq:Sn-split}
			S_n(h)=S_{1,n}(h)+S_{2,n}(h)+S_{3,n}(h),
		\end{equation}
		where
		\[
		S_{1,n}(h)=\sum_{i\in\{\sigma,H,\alpha\}}\delta_{i,n}M_{i,n}(\theta)
		=\frac{1}{\sqrt{n\Delta_n}}\bigl(h_1C_{\sigma,n}+h_2D_{H,n}^{\perp}+h_3A_{\alpha,n}^{\perp}\bigr),
		\]
		\[
		S_{2,n}(h)=\frac12\sum_{i,j}\delta_{i,n}\delta_{j,n}N_{ij,n}(\theta),
		\]
		and
		\[
		S_{3,n}(h)=\sum_{i,j,k}\delta_{i,n}\delta_{j,n}\delta_{k,n}
		\int_0^1\frac{(1-t)^2}{2}\,P_{ijk,n}(\theta+t\delta_n(h))\,\dd t.
		\]
		Moreover,
		\begin{equation}\label{eq:S1-bounds}
			\norm{S_{1,n}(h)}_{\op}\le C_h\frac{\log n}{\sqrt{n\Delta_n}},
			\qquad
			\norm{S_{1,n}(h)}_{\F}=O(1),
			\qquad
			\frac12\tr\bigl(S_{1,n}(h)^2\bigr)\to h^{\top}\Iperp(\theta)h,
		\end{equation}
		and
		\begin{equation}\label{eq:S2-bounds}
			\norm{S_{2,n}(h)}_{\op}\le C_h\frac{\log^2 n}{n\Delta_n},
			\qquad
			\norm{S_{2,n}(h)}_{\F}\le C_h\frac{\log^2 n}{\sqrt{n\Delta_n}},
		\end{equation}
		\begin{equation}\label{eq:S3-bounds}
			\norm{S_{3,n}(h)}_{\op}\le C_h\frac{\log^3 n}{(n\Delta_n)^{3/2}},
			\qquad
			\norm{S_{3,n}(h)}_{\F}\le C_h\frac{\log^3 n}{n\Delta_n}.
		\end{equation}
		In particular,
		\begin{equation}\label{eq:SminusS1}
			\norm{S_n(h)-S_{1,n}(h)}_{\op}+\norm{S_n(h)-S_{1,n}(h)}_{\F}\to0.
		\end{equation}
	\end{proposition}
	\begin{proof}
		Since \(\Sigma_n(\cdot)\) is \(C^3\) in a neighborhood of \(\theta\), Taylor's formula with integral remainder $R_{ijk,n}(\theta,\delta_n)$ gives
		\begin{align*}
			\Sigma_n(\theta+\delta_n)-\Sigma_n(\theta)
			& =
			\sum_i \delta_{i,n}\,\partial_i\Sigma_n(\theta)
			+\frac12\sum_{i,j}\delta_{i,n}\delta_{j,n}\,\partial_{ij}^2\Sigma_n(\theta)\\
			& \quad +\sum_{i,j,k}\delta_{i,n}\delta_{j,n}\delta_{k,n}R_{ijk,n}(\theta,\delta_n)
			,
		\end{align*}
		where $R_{ijk,n}(\theta,\delta_n) =\int_0^1 \frac{(1-t)^2}{2}\,\partial_{ijk}^3\Sigma_n(\theta+t\delta_n)\,\dd t$. After sandwiching by \(\Sigma_n(\theta)^{-1/2}\), this yields the decomposition \eqref{eq:Sn-split}.  
		
		First, we consider $S_{1,n}(h)$ term. The expression for \(S_{1,n}(h)\) is exactly \eqref{eq:clean-first-order}. Using
		\[
		\|C_{\sigma,n}\|_{\op}=O(1),\qquad
		\|D_{H,n}^{\perp}\|_{\op}=O(\log n),\qquad
		\|A_{\alpha,n}^{\perp}\|_{\op}=O(\log n),
		\]
		we obtain the operator bound in \eqref{eq:S1-bounds}. Moreover, by the exact orthogonality
		\[
		\tr(C_{\sigma,n}D_{H,n}^{\perp})
		=
		\tr(C_{\sigma,n}A_{\alpha,n}^{\perp})
		=
		\tr(D_{H,n}^{\perp}A_{\alpha,n}^{\perp})=0,
		\]
		together with the trace limits from Sections~3--4, we obtain
		\[
		\frac12\tr\bigl(S_{1,n}(h)^2\bigr)
		=
		\frac{1}{2n\Delta_n}
		\Bigl(
		h_1^2\tr(C_{\sigma,n}^2)
		+h_2^2\tr((D_{H,n}^{\perp})^2)
		+h_3^2\tr((A_{\alpha,n}^{\perp})^2)
		\Bigr)
		\to h^\top \mathcal{I}^\perp(\theta)h.
		\]
		In particular, \(\|S_{1,n}(h)\|_F=O(1)\).
		
		Next, by lemmas \ref{lem:delta-size} and \ref{lem:higher-deriv-bounds},
		\[
		\|S_{2,n}(h)\|_{\op}
		\le
		\frac12\sum_{i,j}|\delta_{i,n}\delta_{j,n}|\,\|N_{ij,n}(\theta)\|_{\op}
		\le
		C_h\frac{\log^2 n}{n\Delta_n},
		\]
		because \(|\delta_{i,n}|\le C_h/\sqrt{n\Delta_n}\) and the worst case \(\nu_{ij}=2\) contributes an additional factor \((\log n)^2\). Similarly,
		\[
		\|S_{2,n}(h)\|_{F}
		\le
		\frac12\sum_{i,j}|\delta_{i,n}\delta_{j,n}|\,\|N_{ij,n}(\theta)\|_{F}
		\le
		C_h\frac{\log^2 n}{\sqrt{n\Delta_n}}.
		\]
		Both bounds vanish by \eqref{eq:poly-log-dom}, proving \eqref{eq:S2-bounds}.

        Finally, the third-order remainder has Taylor weight \(\int_0^1 (1-t)^2/2\,dt=1/6\).
Since \(S_{3,n}(h)\) is integrated along the segment
\[
\{\theta+t\delta_n(h):\ t\in[0,1]\},
\]
this segment is contained in \(\mathcal U_n(h_0)\) for any fixed \(h_0>\|h\|\). Hence,
using again Lemmas~\ref{lem:delta-size} and \ref{lem:higher-deriv-bounds}, together with \(\nu_{ijk}\le 3\),
\[
\|S_{3,n}(h)\|_{\op}\le C_h\frac{\log^3 n}{(n\Delta_n)^{3/2}},
\qquad
\|S_{3,n}(h)\|_{F}\le C_h\frac{\log^3 n}{n\Delta_n}.
\]	These also vanish by \eqref{eq:poly-log-dom}, so \eqref{eq:S3-bounds} follows. Therefore
		\[
		\|S_n(h)-S_{1,n}(h)\|_{\op}+\|S_n(h)-S_{1,n}(h)\|_{F}\to 0,
		\]
		which is exactly \eqref{eq:SminusS1}. This completes the proof.
	\end{proof} 
	
	\begin{theorem}\label{thm:LAN}
		Assume $H>3/4$, $\Delta_n\to0$, $T_n=n\Delta_n\to\infty$, \eqref{eq:mild-growth}, and the polynomial mesh condition \eqref{eq:poly-mesh}. Let
		\[
		\theta=(\sigma,H,\alpha)
		\qquad\text{with}\qquad
		\sigma>0,\ \alpha>0,\ H\in(3/4,1).
		\]
		For every fixed $h\in\R^3$, define the local alternative
		\[
		\theta_{n,h}=\theta+r_n(\theta)^{-1}h,
		\qquad
		r_n(\theta)=\sqrt{n\Delta_n}\,(M_n^\top)^{-1}.
		\]
		Then
		\begin{equation}\label{eq:LAN-main}
			\ell_n(\theta_{n,h})-\ell_n(\theta)
			=h^\top\Xi_n-\frac12 h^\top\Iperp(\theta)h+o_{P_\theta}(1),
		\end{equation}
		where
		\[
		\Xi_n=\frac1{\sqrt{n\Delta_n}}
		\begin{pmatrix}
			S_{\sigma,n}\\ R_{H,n}^{\perp}\\ S_{\alpha,n}^{\perp}
		\end{pmatrix}
		\dto \mathcal{N}\big(0,\Iperp(\theta)\big),
		\qquad
		\Iperp(\theta)=\diag\Big(I_{\sigma\sigma},I_{HH}^{\perp},I_{\alpha\alpha}^{\perp}\Big)
		\]
		and  $I_{\sigma\sigma}$ in \eqref{eq:Iss}, $I_{HH}^{\perp}$ in \eqref{eq:IHHperp}, $I_{\alpha\alpha}^{\perp}$ in \eqref{eq:Iaa-perp}, $M_n$ in \eqref{eq:Mn-def}.
		In particular, the statistical experiment generated by $X^{(n)}$ is locally asymptotically normal at $\theta$ with non-diagonal rate matrix $r_n(\theta)$.
	\end{theorem}
	\begin{proof}
		We decompose \(S_n(h)\) into its linear part and a negligible remainder, and then expand the Gaussian log-likelihood ratio to second order. Fix \(h\in\R^3\), and write
		\[
		\theta_h:=\theta_{n,h},\qquad
		\Sigma:=\Sigma_n(\theta),\qquad
		\Sigma_h:=\Sigma_n(\theta_h),\qquad
		S:=S_n(h),
		\]
		and
		\[
		Z_n:=\Sigma^{-1/2}X^{(n)}\sim \mathcal{N}(0,I_n)
		\quad\text{under }P_\theta.
		\]
		By construction,
		\[
		\Sigma_h=\Sigma^{1/2}(I_n+S)\Sigma^{1/2},
		\]
		and hence the Gaussian likelihood identity yields
		\begin{equation}\label{eq:llr-exact-proof}
			\ell_n(\theta_h)-\ell_n(\theta)
			=
			-\frac12\log\det(I_n+S)
			-\frac12 Z_n^\top\bigl((I_n+S)^{-1}-I_n\bigr)Z_n.
		\end{equation}
		By proposition~\ref{prop:Sn-split}, 
		\[
		S=S_{1,n}(h)+R_n(h),
		\qquad
		R_n(h):=S_{2,n}(h)+S_{3,n}(h),
		\]
		with
		\[
		\|R_n(h)\|_{\op}+\|R_n(h)\|_{F}\to0.
		\]
		Moreover, by \eqref{eq:S1-bounds},
		\[
		\|S_{1,n}(h)\|_F=O(1),
		\qquad
		\|S_{1,n}(h)\|_{\op}\to0,
		\qquad
		\frac12\tr\bigl(S_{1,n}(h)^2\bigr)\to h^\top \mathcal{I}^\perp(\theta)h.
		\]
		Therefore,
		\[
		\|S\|_{\op}\le \|S_{1,n}(h)\|_{\op}+\|R_n(h)\|_{\op}\to0.
		\]
		We first consider the linear term. Since
		\[
		\frac12\bigl(Z_n^\top SZ_n-\tr(S)\bigr)
		=
		\frac12\bigl(Z_n^\top S_{1,n}(h)Z_n-\tr(S_{1,n}(h))\bigr)
		+\frac12\bigl(Z_n^\top R_n(h)Z_n-\tr(R_n(h))\bigr),
		\]
		the first term is exactly \(h^\top\Xi_n\) by \eqref{eq:clean-first-order}. For the second term,
		\[
		\Var \!\Bigl(\frac12\bigl(Z_n^\top R_n(h)Z_n-\tr(R_n(h))\bigr)\Bigr)
		=
		\frac12\tr\bigl(R_n(h)^2\bigr)
		=
		\frac12\|R_n(h)\|_F^2\to0.
		\]
		Hence
		\begin{equation}\label{eq:first-order-final}
			\frac12\bigl(Z_n^\top SZ_n-\tr(S)\bigr)
			=
			h^\top\Xi_n+o_{P_\theta}(1).
		\end{equation}
		
		Next, expanding \(\tr(S^2)\) gives
		\[
		\tr(S^2)
		=
		\tr\bigl(S_{1,n}(h)^2\bigr)
		+2\tr\bigl(S_{1,n}(h)R_n(h)\bigr)
		+\tr\bigl(R_n(h)^2\bigr).
		\]
		Since \(\|S_{1,n}(h)\|_F=O(1)\) and \(\|R_n(h)\|_F\to0\), the last two terms are \(o(1)\), and therefore
		\[
		\tr(S^2)=\tr\bigl(S_{1,n}(h)^2\bigr)+o(1).
		\]
		Using \eqref{eq:S1-bounds}, we obtain
		\begin{equation}\label{eq:trace-S2-limit}
			\frac14\tr(S^2)\to \frac12\,h^\top \mathcal{I}^\perp(\theta)h.
		\end{equation}
		In particular, \eqref{eq:trace-S2-limit} implies that \(\tr(S^2)=O(1)\).
		
		We now control the random quadratic correction. By Wick's formula,
		\[
		\Var \bigl(Z_n^\top S^2Z_n-\tr(S^2)\bigr)
		=
		2\tr(S^4)
		\le
		2\|S\|_{\op}^2\tr(S^2).
		\]
		Since \(\|S\|_{\op}\to0\) and \(\tr(S^2)=O(1)\), it follows that
		\begin{equation}\label{eq:random-quad-small}
			Z_n^\top S^2Z_n=\tr(S^2)+o_{P_\theta}(1).
		\end{equation}
		Because \(\|S\|_{\op}\to0\), we have \(\|S\|_{\op}\le 1/2\) for all sufficiently large \(n\), so lemmas~\ref{lem:logdet} and \ref{lem:inv} apply. Substituting their second order expansions into \eqref{eq:llr-exact}, we obtain
		\[
		\ell_n(\theta_h)-\ell_n(\theta)
		=
		\frac12\bigl(Z_n^\top SZ_n-\tr(S)\bigr)
		-\frac12 Z_n^\top S^2Z_n
		+\frac14\tr(S^2)
		-\frac12R_{\log}(S)
		-\frac12 Z_n^\top R_{\mathrm{inv}}(S)Z_n.
		\]
		The log-determinant remainder satisfies
		\[
		|R_{\log}(S)|\le C\|S\|_{\op}\tr(S^2)=o(1),
		\]
		and the inverse remainder satisfies
		\[
		|\tr(R_{\mathrm{inv}}(S))|
		\le C\|S\|_{\op}\tr(S^2)=o(1).
		\]
		Moreover,
		\[
		\Var \!\Bigl(Z_n^\top R_{\mathrm{inv}}(S)Z_n-\tr(R_{\mathrm{inv}}(S))\Bigr)
		=
		2\tr\bigl(R_{\mathrm{inv}}(S)^2\bigr)
		\le
		C\|S\|_{\op}^4\tr(S^2)\to0.
		\]
		Hence
		\[
		Z_n^\top R_{\mathrm{inv}}(S)Z_n
		=
		\tr(R_{\mathrm{inv}}(S))+o_{P_\theta}(1).
		\]
		Since \(\tr(R_{\mathrm{inv}}(S))=o(1)\), we conclude that
		\begin{equation}\label{eq:Rinv-small}
			Z_n^\top R_{\mathrm{inv}}(S)Z_n=o_{P_\theta}(1).
		\end{equation}
		
		Combining \eqref{eq:first-order-final}, \eqref{eq:trace-S2-limit},
		\eqref{eq:random-quad-small}, \eqref{eq:Rinv-small}, and the bound
		\(R_{\log}(S)=o(1)\), we conclude that
		\[
		\ell_n(\theta_h)-\ell_n(\theta)
		=
		h^\top\Xi_n-\frac14\tr(S^2)+o_{P_\theta}(1)
		=
		h^\top\Xi_n-\frac12 h^\top \mathcal{I}^\perp(\theta)h+o_{P_\theta}(1).
		\]
		Finally, proposition~\ref{prop:joint-clt} yields
		\[
		\Xi_n \dto \mathcal{N}\bigl(0,\mathcal{I}^\perp(\theta)\bigr),
		\]
		which proves the LAN property.
	\end{proof}
	\begin{remark}
		As always, LAN is invariant under deterministic invertible linear reparametrization of the local parameter. Thus one may map the diagonalized $h$-coordinates back to the original $(\sigma,H,\alpha)$ basis. The price is that the limiting information matrix becomes symmetric but non-diagonal. The lower-triangular form used here is simply the cleanest way to expose the three independent asymptotic directions.
	\end{remark}
	
	\section{Figure design for the three dimensional case}
	In the two-dimensional mfBm paper \cite{Cai2026}, the simulation section can display the full projected Gaussian by a single contour plot and a single surface plot. In the present three dimensional setting, a single picture cannot faithfully display the full limit. A practical and mathematically natural design is therefore:
	\begin{enumerate}[label=(\roman*),itemsep=2pt]
		\item a low frequency profile plot of the weight $w(u)=1/(1+A^{-1}|u|^\rho)$, which explains where the Fisher information comes from;
		\item three pairwise contour plots for the bivariate projected limits
		$(\Xi_{\sigma},\Xi_H^{\perp})$, $(\Xi_{\sigma},\Xi_{\alpha}^{\perp})$, and $(\Xi_H^{\perp},\Xi_{\alpha}^{\perp})$;
		\item optionally, a three dimensional Gaussian ellipsoid or density surface for the diagonal limit in $(\Xi_{\sigma},\Xi_H^{\perp},\Xi_{\alpha}^{\perp})$ coordinates.
	\end{enumerate}
	Since the projected covariance matrix is diagonal, the pairwise contour plots are especially informative: every ellipse is axis-aligned, visually emphasizing that the three step projection has removed the asymptotic correlations.
	
	Below we include two schematic figures that can be compiled directly in \LaTeX{}.
	
	\begin{figure}[htbp]
		\centering
		\begin{tikzpicture}
			\begin{axis}[
				width=0.72\textwidth,
				height=0.36\textwidth,
				xlabel={$u$}, ylabel={$w(u)$},
				legend style={draw=none,fill=none,at={(0.97,0.97)},anchor=north east},
				axis lines=left, domain=-8:8, samples=300, ymin=0, ymax=1.05]
				\addplot[very thick,blue!70!black] {1/(1 + (abs(x)^0.6)/2.5)};
				\addlegendentry{$w(u)=\big(1+A^{-1}|u|^{\rho}\big)^{-1}$}
			\end{axis}
		\end{tikzpicture}
		\caption{Schematic low frequency profile. The information is concentrated on the scale $\lambda\asymp\Delta_n$, hence after the rescaling $\lambda=\Delta_nu$ the effective weight becomes $w(u)$.}
		\label{fig:weight}
	\end{figure}
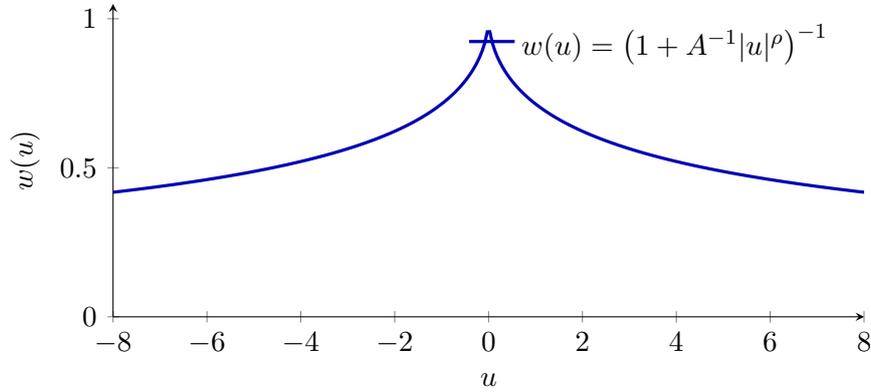
	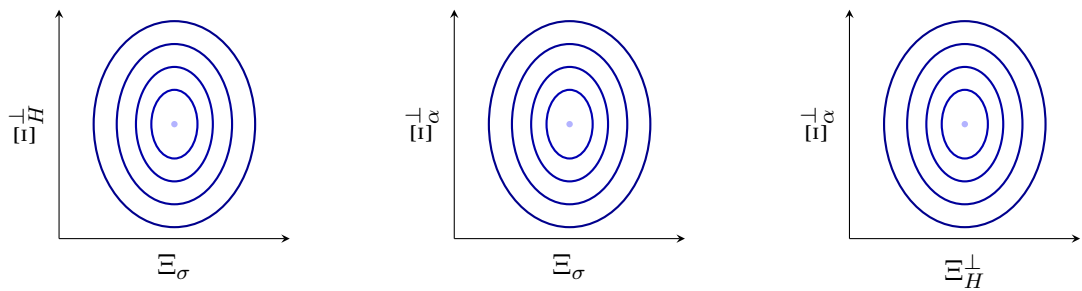
\begin{figure}[htbp]
		\centering
		\begin{tikzpicture}
			\foreach \shift/\xlabel/\ylabel in {
				0/{$\Xi_{\sigma}$}/{$\Xi_H^{\perp}$},
				5.2/{$\Xi_{\sigma}$}/{$\Xi_{\alpha}^{\perp}$},
				10.4/{$\Xi_H^{\perp}$}/{$\Xi_{\alpha}^{\perp}$}}
			{
				\begin{scope}[xshift=\shift cm]
					\begin{axis}[
						width=0.29\textwidth,
						height=0.29\textwidth,
						axis lines=left,
						xtick=\empty, ytick=\empty,
						xmin=-3,xmax=3,ymin=-3,ymax=3,
						xlabel=\xlabel,ylabel=\ylabel]
						\draw[blue!70!black, thick] (axis cs:0,0) ellipse [x radius=0.6, y radius=0.9];
						\draw[blue!65!black, thick] (axis cs:0,0) ellipse [x radius=1.0, y radius=1.5];
						\draw[blue!60!black, thick] (axis cs:0,0) ellipse [x radius=1.5, y radius=2.1];
						\draw[blue!55!black, thick] (axis cs:0,0) ellipse [x radius=2.1, y radius=2.7];
						\fill[blue!30] (axis cs:0,0) circle[radius=1.2pt];
					\end{axis}
				\end{scope}
			}
		\end{tikzpicture}
		\caption{Pairwise contour plots for the projected Gaussian limit. Because the limiting covariance matrix is diagonal in the projected coordinates, the contours are axis-aligned. This is the most informative visual summary in dimension three.}
		\label{fig:pairwise}
	\end{figure}
	\newpage
	\section{Concluding summary}
	The full mfOU LAN proof follows the same conceptual mechanism in \cite{Cai2026}, but one extra parameter forces a third projection step. The essential chain is
	\[
	\begin{aligned}
		\text{raw scores}
		&\longrightarrow \text{remove explicit $\log\Delta_n$ from $H$}\\
		&\longrightarrow \text{project $H$ against $\sigma$}\\
		&\longrightarrow \text{project $\alpha$ against $(\sigma,H^{\perp})$}\\
		&\longrightarrow \text{diagonal CLT}\\
		&\longrightarrow \text{non-diagonal rate matrix and LAN}.
	\end{aligned}
	\]
	The rate matrix is not guessed in advance. It is read directly from the exact score identities. This is the cleanest way to justify the op/F approach and to see why the final normalization is lower triangular rather than diagonal.

	\appendix
	\section{Why the rate matrix is non-diagonal}
	The mfBm paper \cite{Cai2026} stresses that in the supercritical regime the natural local scale is not obtained by a diagonal normalization in the original parameter coordinates. The same happens here. Indeed,
	\[
	S_{H,n}=\sigma\log\Delta_n\,S_{\sigma,n}+R_{H,n},
	\]
	so the raw \(H\)-score contains an explicit singular component in the \(\sigma\)-direction. Writing
	\[
	R_{H,n}=a_nS_{\sigma,n}+R_{H,n}^{\perp},
	\]
	we obtain
	\[
	S_{H,n}=\beta_n(\theta)S_{\sigma,n}+R_{H,n}^{\perp},
	\qquad
	\beta_n(\theta):=\sigma\log\Delta_n+a_n.
	\]
	Thus a diagonal normalization in the original \((\sigma,H,\alpha)\) coordinates does not separate the effective score directions. The lower-triangular matrix \(M_n\) removes this collinearity.
	
	After that, the $\alpha$-direction still has non-zero covariance with the first two directions and the second projection $M_n^{(3)}$ removes those terms. In this sense the matrix $M_n$ is not an arbitrary algebraic device. It is the canonical implementation of Gram--Schmidt orthogonalization for the three raw score directions.
	
	\section{Simulation blueprint in the style of \cite{Cai2026}}
	To extend the current manuscript further toward the paper \cite{Cai2026} , the natural simulation section should contain the following three blocks.
	
	\paragraph{Projected full-rank limit.}
	Fix one supercritical point, for instance $(\sigma,H,\alpha)=(1,0.8,1)$, simulate the exact Gaussian vector $X^{(n)}$ from its Toeplitz covariance, compute the transformed score
	\[
	\Xi_n=\frac{1}{\sqrt{n\Delta_n}}M_n\nabla\ell_n(\theta),
	\]
	and compare the empirical covariance with the diagonal limit $\Iperp(\theta)$.
	
	\paragraph{Pairwise Gaussian diagnostics.}
	Plot the three empirical bivariate clouds for the pairs
	$(\Xi_{\sigma},\Xi_H^{\perp})$, $(\Xi_{\sigma},\Xi_{\alpha}^{\perp})$, and $(\Xi_H^{\perp},\Xi_{\alpha}^{\perp})$.
	Overlay each cloud with the corresponding Gaussian contour. Since the limiting covariance is diagonal, every contour should be axis-aligned.
	
	\paragraph{Raw versus projected comparison.}
	A particularly instructive figure, not present in the two-dimensional mfBm case, is to compare the raw pair $(S_{\sigma,n},S_{H,n})$ with the projected pair $(S_{\sigma,n},R_{H,n}^{\perp})$. The first plot should display the near-collinearity caused by the diverging $\log\Delta_n$ term, while the second should recover a full-rank cloud. This figure directly visualizes why the non-diagonal rate matrix is necessary.

	\section{\texorpdfstring{The regime $1/2<H<3/4$: complete score CLT in the subcritical long-memory case}{The regime 1/2<H<3/4: complete score CLT in the subcritical long-memory case}}
	In this appendix we parallel Appendix~C of the \cite{Cai2026}, but with the mfOU sampled spectral density and the extra drift parameter $\alpha$. Throughout this section we assume
	\[
	\frac12<H<\frac34,
	\qquad
	p:=2H-1\in(0,1/2),
	\qquad
	\Delta_n=n^{-\kappa},\ \kappa\in(0,1).
	\]
	Then the exact score decomposition
	\[
	S_{H,n}=\sigma\log\Delta_n\,S_{\sigma,n}+R_{H,n}
	\]
	remains valid. In contrast with the supercritical regime $H>3/4$, after removing the explicit deterministic $\sigma\log\Delta_n$ term no second projection is needed for the $(\sigma,H)$-block: the pair $(S_{\sigma,n},R_{H,n})$ already has a non-degenerate Gaussian limit. The $\alpha$-score lives on the smaller scale $(n\Delta_n)^{1/2}$ and becomes asymptotically orthogonal, after normalization, to the first two coordinates.
	
	\subsection*{C.1. Trace approximations at the scale $n\Delta_n^{2p}$ for the $(\sigma,H)$-block}
	Let
	\[
	M_{\sigma,n}:=\Sigma_n^{-1/2}\partial_\sigma\Sigma_n\Sigma_n^{-1/2},
	\qquad
	D_{H,n}:=\Sigma_n^{-1/2}T_n(r_{\Delta_n})\Sigma_n^{-1/2},
	\]
	so that
	\[
	S_{\sigma,n}=\frac12Q_n(M_{\sigma,n}),
	\qquad
	R_{H,n}=\frac12Q_n(D_{H,n}).
	\]
	Write the symbol ratios
	\[
	g_{\sigma,n}(\lambda):=\frac{\partial_\sigma f_{\Delta_n}(\lambda;\theta)}{f_{\Delta_n}(\lambda;\theta)},
	\qquad
	h_n(\lambda):=\frac{r_{\Delta_n}(\lambda;\theta)}{f_{\Delta_n}(\lambda;\theta)}.
	\]
	In the regime $1/2<H<3/4$ the natural fixed-frequency renormalization is by $\Delta_n^p$. Define
	\[
	\bar g_n(\lambda):=\Delta_n^{-p}g_{\sigma,n}(\lambda),
	\qquad
	\bar h_n(\lambda):=\Delta_n^{-p}h_n(\lambda).
	\]
	The mfOU low frequency expansion gives, uniformly for $\lambda\in[-\pi,\pi]\setminus\{0\}$,
	\begin{equation}\label{eq:subcritical-envelope}
		|\bar g_n(\lambda)|\le C\bigl(1+|\lambda|^{-p}\bigr),
		\qquad
		|\bar h_n(\lambda)|\le C\bigl(1+|\lambda|^{-p}(1+|\log|\lambda||)\bigr).
	\end{equation}
	Because $p<1/2$, both envelopes belong to $L^2([ -\pi,\pi])$.
	
	\begin{lemma}\label{lem:subcritical-trace}
		Assume $1/2<H<3/4$. Then
		\[
		\tr(M_{\sigma,n}^2)
		=\frac{n\Delta_n^{2p}}{2\pi}\int_{-\pi}^{\pi}\bar g_n(\lambda)^2\,\dd\lambda+o(n\Delta_n^{2p}),
		\]
		\[
		\tr(M_{\sigma,n}D_{H,n})
		=\frac{n\Delta_n^{2p}}{2\pi}\int_{-\pi}^{\pi}\bar g_n(\lambda)\bar h_n(\lambda)\,\dd\lambda+o(n\Delta_n^{2p}),
		\]
		\[
		\tr(D_{H,n}^2)
		=\frac{n\Delta_n^{2p}}{2\pi}\int_{-\pi}^{\pi}\bar h_n(\lambda)^2\,\dd\lambda+o(n\Delta_n^{2p}).
		\]
	\end{lemma}
	
	\begin{proof}
		Set
		\[
		B_{\sigma,n}:=\Sigma_n^{-1}T_n(\partial_\sigma f_{\Delta_n}),
		\qquad
		B_{H,n}:=\Sigma_n^{-1}T_n(r_{\Delta_n}).
		\]
		Since $M_{\sigma,n}$ is similar to $B_{\sigma,n}$ and $D_{H,n}$ is similar to $B_{H,n}$, cyclicity of the trace reduces the problem to products of $B_{\sigma,n}$ and $B_{H,n}$. Now write
		\[
		G_n:=T_n(\bar g_n),
		\qquad
		H_n:=T_n(\bar h_n).
		\]
		The same Toeplitz sandwich reduction used in the proof of lemma~\ref{lem:C-trace} applies here as well: after factoring out $\Delta_n^p$, the difference between the whitened products $B_{\sigma,n}^2$, $B_{\sigma,n}B_{H,n}$, $B_{H,n}^2$ and the pure Toeplitz products $G_n^2$, $G_nH_n$, $H_n^2$ contributes only $o(n)$ to the trace. Consequently,
		\[
		\tr(M_{\sigma,n}^2)=\Delta_n^{2p}\tr(G_n^2)+o(n\Delta_n^{2p}),
		\]
		\[
		\tr(M_{\sigma,n}D_{H,n})=\Delta_n^{2p}\tr(G_nH_n)+o(n\Delta_n^{2p}),
		\qquad
		\tr(D_{H,n}^2)=\Delta_n^{2p}\tr(H_n^2)+o(n\Delta_n^{2p}).
		\]
		Because $p\in(0,1/2)$, the envelopes in \eqref{eq:subcritical-envelope} lie in $L^2([{-}\pi,\pi])$. Hence the triangular-array Avram trace theorem applies to the three pairs $(\bar g_n,\bar g_n)$, $(\bar g_n,\bar h_n)$, and $(\bar h_n,\bar h_n)$:
		\[
		\tr(G_n^2)=\frac{n}{2\pi}\int_{-\pi}^{\pi}\bar g_n(\lambda)^2\,\dd\lambda+o(n),
		\]
		\[
		\tr(G_nH_n)=\frac{n}{2\pi}\int_{-\pi}^{\pi}\bar g_n(\lambda)\bar h_n(\lambda)\,\dd\lambda+o(n),
		\qquad
		\tr(H_n^2)=\frac{n}{2\pi}\int_{-\pi}^{\pi}\bar h_n(\lambda)^2\,\dd\lambda+o(n).
		\]
		Multiplying back the factor $\Delta_n^{2p}$ proves the three expansions.
	\end{proof}
	
	\begin{lemma}\label{lem:subcritical-integrals}
		Assume $1/2<H<3/4$. Then there exist functions $\bar g^{(<)}_{\sigma},\bar h^{(<)}\in L^2([ -\pi,\pi])$ such that
		\[
		\bar g_n\to \bar g^{(<)}_{\sigma},
		\qquad
		\bar h_n\to \bar h^{(<)}
		\qquad\text{in }L^2([ -\pi,\pi]),
		\]
		and hence
		\[
		\frac{1}{n\Delta_n^{2p}}\tr(M_{\sigma,n}^2)\to K_{\sigma\sigma}^{(<)}(\theta),
		\qquad
		\frac{1}{n\Delta_n^{2p}}\tr(M_{\sigma,n}D_{H,n})\to K_{\sigma H}^{(<)}(\theta),
		\]
		\[
		\frac{1}{n\Delta_n^{2p}}\tr(D_{H,n}^2)\to K_{HH}^{(<)}(\theta),
		\]
		where
		\[
		K_{\sigma\sigma}^{(<)}(\theta):=\frac{1}{2\pi}\int_{-\pi}^{\pi}\bigl(\bar g^{(<)}_{\sigma}(\lambda)\bigr)^2\,\dd\lambda,
		\qquad
		K_{\sigma H}^{(<)}(\theta):=\frac{1}{2\pi}\int_{-\pi}^{\pi}\bar g^{(<)}_{\sigma}(\lambda)\bar h^{(<)}(\lambda)\,\dd\lambda,
		\]
		\[
		K_{HH}^{(<)}(\theta):=\frac{1}{2\pi}\int_{-\pi}^{\pi}\bigl(\bar h^{(<)}(\lambda)\bigr)^2\,\dd\lambda.
		\]
		Moreover,
		\[
		K_{\sigma\sigma}^{(<)}(\theta)>0,
		\qquad
		K_{HH}^{(<)}(\theta)>0,
		\qquad
		K_{\sigma\sigma}^{(<)}(\theta)K_{HH}^{(<)}(\theta)-\bigl(K_{\sigma H}^{(<)}(\theta)\bigr)^2>0.
		\]
	\end{lemma}
	
	\begin{proof}
		For every fixed $\lambda\neq0$, the Brownian OU part contributes only a regular perturbation while the fractional part produces the leading fixed-frequency factor $\Delta_n^p|\lambda|^{-p}$. Consequently $\bar g_n(\lambda)$ and $\bar h_n(\lambda)$ converge pointwise to limits $\bar g^{(<)}_{\sigma}(\lambda)$ and $\bar h^{(<)}(\lambda)$. The $L^2$ convergence follows from dominated convergence using \eqref{eq:subcritical-envelope}. The trace limits then follow from lemma \ref{lem:subcritical-trace}. Finally, strict positivity and non-degeneracy are a direct consequence of the fact that the two limiting functions are not collinear: $\bar h^{(<)}$ contains the logarithmic factor inherited from the $H$-derivative, whereas $\bar g^{(<)}_{\sigma}$ does not.
	\end{proof}
	
	\subsection*{C.2. Operator/Frobenius ratios in the subcritical regime}
	Define the subcritical normalization
	\[
	v_n:=\sqrt n\,\Delta_n^p.
	\]
	This is the exact analogue of the normalization used in Appendix~C of \cite{Cai2026}, the only difference being that the limiting constants now depend on $(\alpha,\sigma,H)$ through the mfOU filter.
	
	\begin{lemma}\label{lem:subcritical-opf}
		Assume $1/2<H<3/4$. Then
		\[
		\norm{M_{\sigma,n}}_{\op}=O(\Delta_n^p n^p),
		\qquad
		\norm{D_{H,n}}_{\op}=O(\Delta_n^p n^p\log n),
		\]
		while
		\[
		\norm{M_{\sigma,n}}_{\F}^2\asymp n\Delta_n^{2p},
		\qquad
		\norm{D_{H,n}}_{\F}^2\asymp n\Delta_n^{2p}.
		\]
		Consequently,
		\[
		\frac{\norm{M_{\sigma,n}}_{\op}}{\norm{M_{\sigma,n}}_{\F}}\to0,
		\qquad
		\frac{\norm{D_{H,n}}_{\op}}{\norm{D_{H,n}}_{\F}}\to0.
		\]
	\end{lemma}
	
	\begin{proof}
		The renormalized symbols $\bar g_n$ and $\bar h_n$ are of Fisher--Hartwig type with exponent $p\in(0,1/2)$. Standard Toeplitz operator-norm bounds therefore give
		\[
		\norm{T_n(\bar g_n)}_{\op}=O(n^p),
		\qquad
		\norm{T_n(\bar h_n)}_{\op}=O(n^p\log n),
		\]
		with the extra $\log n$ coming from the logarithmic factor in the $H$-derivative; compare \cite[Ch.~5]{BottcherSilbermann2006} and \cite{Gray2006}. Passing from these Toeplitz matrices back to $M_{\sigma,n}$ and $D_{H,n}$ is done exactly as in the main text through the generalized Rayleigh quotient in remark~\ref{rem:gen-rayleigh-whitened}: after the deterministic factor $\Delta_n^p$ is restored, the same symbol-ratio bounds give
		\[
		\norm{M_{\sigma,n}}_{\op}=O(\Delta_n^pn^p),
		\qquad
		\norm{D_{H,n}}_{\op}=O(\Delta_n^pn^p\log n).
		\]
		The Frobenius asymptotics follow from lemmas \ref{lem:subcritical-trace} and \ref{lem:subcritical-integrals}. Therefore
		\[
		\frac{\norm{M_{\sigma,n}}_{\op}}{\norm{M_{\sigma,n}}_{\F}}
		\lesssim \frac{\Delta_n^pn^p}{\sqrt n\,\Delta_n^p}=n^{p-1/2}\to0,
		\]
		\[
		\frac{\norm{D_{H,n}}_{\op}}{\norm{D_{H,n}}_{\F}}
		\lesssim \frac{\Delta_n^pn^p\log n}{\sqrt n\,\Delta_n^p}=n^{p-1/2}\log n\to0,
		\]
		because $p<1/2$.
	\end{proof}
	
	\subsection*{C.3. Joint CLT for the transformed $(\sigma,H)$-scores}
	As in the main text, the explicit deterministic contribution in the raw $H$-score is removed by the triangular matrix
	\[
	M_n^{(<,1)}:=
	\begin{pmatrix}
		1&0\\
		-\sigma\log\Delta_n&1
	\end{pmatrix},
	\qquad
	M_n^{(<,1)}
	\begin{pmatrix}
		S_{\sigma,n}\\ S_{H,n}
	\end{pmatrix}
	=
	\begin{pmatrix}
		S_{\sigma,n}\\ R_{H,n}
	\end{pmatrix}.
	\]
	
	\begin{proposition}\label{prop:subcritical-block}
		Assume $1/2<H<3/4$. Then
		\[
		\frac1{v_n}
		\begin{pmatrix}
			S_{\sigma,n}\\ R_{H,n}
		\end{pmatrix}
		\dto N\bigl(0,\mathcal I_{\sigma,H}^{(<)}(\theta)\bigr),
		\]
		where
		\[
		\mathcal I_{\sigma,H}^{(<)}(\theta)
		:=\frac12
		\begin{pmatrix}
			K_{\sigma\sigma}^{(<)}(\theta) & K_{\sigma H}^{(<)}(\theta)\\
			K_{\sigma H}^{(<)}(\theta) & K_{HH}^{(<)}(\theta)
		\end{pmatrix}
		\]
		is positive definite. Equivalently,
		\[
		\frac1{v_n}M_n^{(<,1)}
		\begin{pmatrix}
			S_{\sigma,n}\\ S_{H,n}
		\end{pmatrix}
		\dto \mathcal{N}\bigl(0,\mathcal I_{\sigma,H}^{(<)}(\theta)\bigr).
		\]
	\end{proposition}
	
	\begin{proof}
		Fix $u=(u_1,u_2)\in\R^2$ and define
		\[
		A_n(u):=u_1M_{\sigma,n}+u_2D_{H,n}.
		\]
		Then
		\[
		u_1S_{\sigma,n}+u_2R_{H,n}=\frac12Q_n(A_n(u)).
		\]
		By lemma \ref{lem:subcritical-opf},
		\[
		\frac{\norm{A_n(u)}_{\op}}{\norm{A_n(u)}_{\F}}\to0.
		\]
		Therefore the Gaussian quadratic-form CLT applies to every fixed linear combination. Bilinearity together with lemma \ref{lem:subcritical-integrals} identifies the limiting variance, and the Cram\'er--Wold device yields the joint convergence.
	\end{proof}
	
	\subsection*{C.4. The drift score and the complete three parameter score vector}
	The $\alpha$-direction is more regular and keeps the same $(n\Delta_n)^{1/2}$ scale as in the main text. Let
	\[
	A_{\alpha,n}:=\Sigma_n^{-1/2}\partial_\alpha\Sigma_n\Sigma_n^{-1/2},
	\qquad
	g_{\alpha,n}(\lambda):=\frac{\partial_\alpha f_{\Delta_n}(\lambda;\theta)}{f_{\Delta_n}(\lambda;\theta)}.
	\]
	The low frequency mfOU filter gives, uniformly in $\lambda\in[-\pi,\pi]$,
	\begin{equation}\label{eq:alpha-envelope-subcritical}
		|g_{\alpha,n}(\lambda)|\le \frac{C}{1+(\lambda/\Delta_n)^2},
		\qquad
		g_{\alpha,n}(\Delta_n u)\to q_\alpha(u):=-\frac{2\alpha}{\alpha^2+u^2}.
	\end{equation}
	In particular,
	\[
	\int_{-\pi}^{\pi}g_{\alpha,n}(\lambda)^2\,\dd\lambda\sim \Delta_n\int_{\R}q_\alpha(u)^2\,\dd u=\frac{2\pi\Delta_n}{\alpha}.
	\]
	
	\begin{lemma}\label{lem:subcritical-alpha}
		Assume $1/2<H<3/4$. Then
		\[
		\norm{A_{\alpha,n}}_{\op}=O(1),
		\qquad
		\tr(A_{\alpha,n}^2)=\frac{n}{2\pi}\int_{-\pi}^{\pi}g_{\alpha,n}(\lambda)^2\,\dd\lambda+o(n\Delta_n),
		\]
		and hence
		\[
		\tr(A_{\alpha,n}^2)\sim \frac{n\Delta_n}{\alpha},
		\qquad
		\frac{S_{\alpha,n}}{\sqrt{n\Delta_n}}\dto \mathcal{N}\Bigl(0,\frac{1}{2\alpha}\Bigr).
		\]
		Moreover, the mixed traces satisfy
		\begin{equation}\label{eq:subcritical-cross-orders}
			\tr(M_{\sigma,n}A_{\alpha,n})=O(n\Delta_n),
			\qquad
			\tr(D_{H,n}A_{\alpha,n})=O\bigl(n\Delta_n|\log\Delta_n|\bigr),
		\end{equation}
		so in particular
		\[
		\tr(M_{\sigma,n}A_{\alpha,n})=o\bigl(v_n\sqrt{n\Delta_n}\bigr),
		\qquad
		\tr(D_{H,n}A_{\alpha,n})=o\bigl(v_n\sqrt{n\Delta_n}\bigr).
		\]
	\end{lemma}
	
	\begin{proof}
		The variance statement is proved exactly as in lemma \ref{lem:alpha-raw}: the OU filter localizes $g_{\alpha,n}$ on the region $|\lambda|\asymp\Delta_n$, the operator norm stays bounded, and the Fej\'er-kernel trace approximation yields
		\[
		\tr(A_{\alpha,n}^2)=\frac{n}{2\pi}\int_{-\pi}^{\pi}g_{\alpha,n}(\lambda)^2\,\dd\lambda+o(n\Delta_n)\sim \frac{n\Delta_n}{\alpha}.
		\]
		Hence $S_{\alpha,n}/\sqrt{n\Delta_n}\dto \mathcal{N}(0,1/(2\alpha))$ by lemma \ref{lem:qf-clt}.
		
		For the mixed traces, combine \eqref{eq:subcritical-envelope} and \eqref{eq:alpha-envelope-subcritical}. Using the low frequency change of variables $\lambda=\Delta_n u$, one obtains
		\[
		\int_{-\pi}^{\pi}|g_{\sigma,n}(\lambda)g_{\alpha,n}(\lambda)|\,\dd\lambda
		\le C\Delta_n^p\int_{-\pi}^{\pi}\frac{1+|\lambda|^{-p}}{1+(\lambda/\Delta_n)^2}\,\dd\lambda
		\le C\Delta_n^{p+1}\int_{\R}\frac{1+\Delta_n^{-p}|u|^{-p}}{1+u^2}\,\dd u
		=O(\Delta_n),
		\]
		and similarly
		\[
		\int_{-\pi}^{\pi}|h_n(\lambda)g_{\alpha,n}(\lambda)|\,\dd\lambda
		\le C\Delta_n^p\int_{-\pi}^{\pi}\frac{1+|\lambda|^{-p}(1+|\log|\lambda||)}{1+(\lambda/\Delta_n)^2}\,\dd\lambda
		=O\bigl(\Delta_n|\log\Delta_n|\bigr).
		\]
		These bounds place the mixed symbol products in $L^1([{-}\pi,\pi])$ uniformly in $n$. Therefore the same triangular-array trace approximation used in lemma~\ref{lem:subcritical-trace} applies to the pairs $(\bar g_n,g_{\alpha,n})$ and $(\bar h_n,g_{\alpha,n})$, giving \eqref{eq:subcritical-cross-orders}. Since
		\[
		v_n\sqrt{n\Delta_n}=n\Delta_n^{p+1/2},
		\]
		we have
		\[
		\frac{n\Delta_n}{n\Delta_n^{p+1/2}}=\Delta_n^{1/2-p}\to0,
		\qquad
		\frac{n\Delta_n|\log\Delta_n|}{n\Delta_n^{p+1/2}}=\Delta_n^{1/2-p}|\log\Delta_n|\to0,
		\]
		because $p<1/2$.
	\end{proof}
	
	\begin{proposition}\label{prop:subcritical-full}
		Assume $1/2<H<3/4$. Then
		\[
		\begin{pmatrix}
			S_{\sigma,n}/v_n\\[1mm]
			R_{H,n}/v_n\\[1mm]
			S_{\alpha,n}/\sqrt{n\Delta_n}
		\end{pmatrix}
		\dto \mathcal{N}\Biggl(0,
		\begin{pmatrix}
			\mathcal I_{\sigma,H}^{(<)}(\theta) & 0\\
			0 & \dfrac{1}{2\alpha}
		\end{pmatrix}
		\Biggr).
		\]
		Equivalently,
		\[
		\diag(v_n,v_n,\sqrt{n\Delta_n})^{-1}
		\begin{pmatrix}
			1&0&0\\
			-\sigma\log\Delta_n&1&0\\
			0&0&1
		\end{pmatrix}
		\begin{pmatrix}
			S_{\sigma,n}\\ S_{H,n}\\ S_{\alpha,n}
		\end{pmatrix}
		\dto \mathcal{N}\Biggl(0,
		\begin{pmatrix}
			\mathcal I_{\sigma,H}^{(<)}(\theta) & 0\\
			0 & \dfrac{1}{2\alpha}
		\end{pmatrix}
		\Biggr).
		\]
	\end{proposition}
	
	\begin{proof}
		The first two coordinates converge by proposition \ref{prop:subcritical-block}, the third by lemma \ref{lem:subcritical-alpha}. The normalized cross-covariances vanish by lemma \ref{lem:subcritical-alpha}, and a final application of Cram\'er--Wold concludes the proof.
	\end{proof}
	
	\section{\texorpdfstring{The regime $0<H<1/2$: complete score CLT in the fBm-dominated normalization}{The regime 0<H<1/2: complete score CLT in the fBm-dominated normalization}}
	We now parallel Appendix~D of \cite{Cai2026}. In this regime the natural small parameter is
	\[
	\varepsilon_n:=\Delta_n^{1-2H}\downarrow0,
	\]
	and we factor the covariance as
	\[
	\Sigma_n(\theta)=\Delta_n^{2H}B_n(\theta),
	\qquad
	B_n(\theta):=T_n\bigl(\sigma^2\widetilde f_{\Delta_n}^{(H)}(\cdot;\alpha,H)+\varepsilon_n\widetilde f_{\Delta_n}^{(W)}(\cdot;\alpha)\bigr).
	\]
	Under $P_\theta$, define the whitened vector
	\[
	\widetilde Z_n:=\Sigma_n(\theta)^{-1/2}X^{(n)}=\Delta_n^{-H}B_n(\theta)^{-1/2}X^{(n)}\sim N(0,I_n).
	\]
	As before,
	\[
	S_{H,n}=\sigma\log\Delta_n\,S_{\sigma,n}+R_{H,n}.
	\]
	
	\subsection*{D.1. Rescaled score matrices and trace asymptotics}
	Define
	\[
	\widetilde C_{\sigma,n}:=B_n^{-1/2}\widetilde T_n(H)B_n^{-1/2},
	\qquad
	\widetilde D_{H,n}:=B_n^{-1/2}\partial_H\widetilde T_n(H)B_n^{-1/2},
	\]
	so that
	\[
	S_{\sigma,n}=\sigma Q_n(\widetilde C_{\sigma,n}),
	\qquad
	R_{H,n}=\frac{\sigma^2}{2}Q_n(\widetilde D_{H,n}).
	\]
	Write the rescaled symbol ratios
	\[
	\widetilde c_n(\lambda):=\frac{\widetilde f_{\Delta_n}^{(H)}(\lambda;\alpha,H)}{\sigma^2\widetilde f_{\Delta_n}^{(H)}(\lambda;\alpha,H)+\varepsilon_n\widetilde f_{\Delta_n}^{(W)}(\lambda;\alpha)},
	\]
	\[
	\widetilde d_n(\lambda):=\widetilde c_n(\lambda)\,\partial_H\log \widetilde f_{\Delta_n}^{(H)}(\lambda;\alpha,H).
	\]
	Because $H<1/2$, the rescaled fractional symbol is bounded and strictly positive at the origin; moreover,
	\begin{equation}\label{eq:lowH-envelope}
		0\le \widetilde c_n(\lambda)\le C,
		\qquad
		|\widetilde d_n(\lambda)|\le C\bigl(1+|\log|\lambda||\bigr),
		\qquad \lambda\in[-\pi,\pi].
	\end{equation}
	Thus both $\widetilde c_n$ and $\widetilde d_n$ lie in $L^2([ -\pi,\pi])$ uniformly in $n$.
	
	\begin{lemma}\label{lem:lowH-trace}
		Assume $0<H<1/2$. Then
		\[
		\tr(\widetilde C_{\sigma,n}^2)=\frac{n}{2\pi}\int_{-\pi}^{\pi}\widetilde c_n(\lambda)^2\,\dd\lambda+o(n),
		\]
		\[
		\tr(\widetilde C_{\sigma,n}\widetilde D_{H,n})=\frac{n}{2\pi}\int_{-\pi}^{\pi}\widetilde c_n(\lambda)\widetilde d_n(\lambda)\,\dd\lambda+o(n),
		\]
		\[
		\tr(\widetilde D_{H,n}^2)=\frac{n}{2\pi}\int_{-\pi}^{\pi}\widetilde d_n(\lambda)^2\,\dd\lambda+o(n).
		\]
		Moreover,
		\[
		\int_{-\pi}^{\pi}\widetilde c_n(\lambda)^2\,\dd\lambda\to \widetilde K_{\sigma\sigma}(\theta),
		\qquad
		\int_{-\pi}^{\pi}\widetilde c_n(\lambda)\widetilde d_n(\lambda)\,\dd\lambda\to \widetilde K_{\sigma H}(\theta),
		\]
		\[
		\int_{-\pi}^{\pi}\widetilde d_n(\lambda)^2\,\dd\lambda\to \widetilde K_{HH}(\theta),
		\]
		with finite constants depending smoothly on $(\alpha,\sigma,H)$.
	\end{lemma}
	
	\begin{proof}
		Let
		\[
		\widetilde B_{\sigma,n}:=B_n^{-1}\widetilde T_n(H),
		\qquad
		\widetilde B_{H,n}:=B_n^{-1}\partial_H\widetilde T_n(H).
		\]
		Then $\widetilde C_{\sigma,n}$ and $\widetilde D_{H,n}$ are similar to $\widetilde B_{\sigma,n}$ and $\widetilde B_{H,n}$, respectively. As in the proof of lemma~\ref{lem:C-trace}, the Toeplitz sandwich reduction transfers the trace calculation from these whitened matrices to the pure Toeplitz matrices $T_n(\widetilde c_n)$ and $T_n(\widetilde d_n)$ with an $o(n)$ error. Therefore it suffices to analyze the latter products.
		
		Because $H<1/2$, the envelope \eqref{eq:lowH-envelope} places both $\widetilde c_n$ and $\widetilde d_n$ uniformly in $L^2([{-}\pi,\pi])$. The triangular-array Avram theorem then yields
		\[
		\tr(T_n(\widetilde c_n)^2)=\frac{n}{2\pi}\int_{-\pi}^{\pi}\widetilde c_n(\lambda)^2\,\dd\lambda+o(n),
		\]
		\[
		\tr(T_n(\widetilde c_n)T_n(\widetilde d_n))=\frac{n}{2\pi}\int_{-\pi}^{\pi}\widetilde c_n(\lambda)\widetilde d_n(\lambda)\,\dd\lambda+o(n),
		\]
		\[
		\tr(T_n(\widetilde d_n)^2)=\frac{n}{2\pi}\int_{-\pi}^{\pi}\widetilde d_n(\lambda)^2\,\dd\lambda+o(n).
		\]
		Transferring these identities back to $\widetilde C_{\sigma,n}$ and $\widetilde D_{H,n}$ gives the displayed trace approximations.
		
		For the limits, note that $\varepsilon_n\to0$ implies pointwise convergence of $\widetilde c_n$ and $\widetilde d_n$ to the corresponding pure-fractional ratios. Dominated convergence applies because of \eqref{eq:lowH-envelope}, producing the finite constants $\widetilde K_{\sigma\sigma}(\theta)$, $\widetilde K_{\sigma H}(\theta)$, and $\widetilde K_{HH}(\theta)$.
	\end{proof}
	
	\subsection*{D.2. Operator/Frobenius ratios and the $(\sigma,H)$-block CLT}
	
	\begin{lemma}\label{lem:lowH-opf}
		Assume $0<H<1/2$. Then
		\[
		\frac{\norm{\widetilde C_{\sigma,n}}_{\op}}{\norm{\widetilde C_{\sigma,n}}_{\F}}\to0,
		\qquad
		\frac{\norm{\widetilde D_{H,n}}_{\op}}{\norm{\widetilde D_{H,n}}_{\F}}\to0.
		\]
	\end{lemma}
	
	\begin{proof}
		For the operator norms, apply remark~\ref{rem:gen-rayleigh-whitened} with $A_n=\widetilde T_n(H)$ and then with $A_n=\partial_H\widetilde T_n(H)$. This gives
		\[
		\norm{\widetilde C_{\sigma,n}}_{\op}\le \sup_{\lambda\in[-\pi,\pi]}|\widetilde c_n(\lambda)|\le C.
		\]
		Likewise,
		\[
		\norm{\widetilde D_{H,n}}_{\op}\le \sup_{\lambda\in[-\pi,\pi]}|\widetilde d_n(\lambda)|.
		\]
		The logarithmic factor in \eqref{eq:lowH-envelope} is truncated at frequencies of order $1/n$, so the corresponding Toeplitz operator norm is $O(\log n)$. Hence
		\[
		\norm{\widetilde D_{H,n}}_{\op}=O(\log n).
		\]
		On the other hand, lemma \ref{lem:lowH-trace} yields
		\[
		\norm{\widetilde C_{\sigma,n}}_{\F}^2\asymp n,
		\qquad
		\norm{\widetilde D_{H,n}}_{\F}^2\asymp n,
		\]
		so
		\[
		\frac{\norm{\widetilde C_{\sigma,n}}_{\op}}{\norm{\widetilde C_{\sigma,n}}_{\F}}\lesssim n^{-1/2}\to0,
		\qquad
		\frac{\norm{\widetilde D_{H,n}}_{\op}}{\norm{\widetilde D_{H,n}}_{\F}}\lesssim \frac{\log n}{\sqrt n}\to0.
		\]
		This proves the lemma.
	\end{proof}
	
	\begin{proposition}\label{prop:lowH-block}
		Assume $0<H<1/2$. Then
		\[
		\frac1{\sqrt n}
		\begin{pmatrix}
			S_{\sigma,n}\\ R_{H,n}
		\end{pmatrix}
		\dto N\bigl(0,\mathcal I_{\sigma,H}^{(0)}(\theta)\bigr),
		\]
		where
		\[
		\mathcal I_{\sigma,H}^{(0)}(\theta)
		:=
		\begin{pmatrix}
			\dfrac{\sigma^2}{\pi}\widetilde K_{\sigma\sigma}(\theta) & \dfrac{\sigma^3}{2\pi}\widetilde K_{\sigma H}(\theta)\\[2mm]
			\dfrac{\sigma^3}{2\pi}\widetilde K_{\sigma H}(\theta) & \dfrac{\sigma^4}{4\pi}\widetilde K_{HH}(\theta)
		\end{pmatrix}.
		\]
		Equivalently,
		\[
		\frac1{\sqrt n}
		\begin{pmatrix}
			1&0\\
			-\sigma\log\Delta_n&1
		\end{pmatrix}
		\begin{pmatrix}
			S_{\sigma,n}\\ S_{H,n}
		\end{pmatrix}
		\dto \mathcal{N}\bigl(0,\mathcal I_{\sigma,H}^{(0)}(\theta)\bigr).
		\]
	\end{proposition}
	
	\begin{proof}
		For each $u=(u_1,u_2)\in\R^2$, define
		\[
		A_n(u):=u_1\sigma\widetilde C_{\sigma,n}+u_2\frac{\sigma^2}{2}\widetilde D_{H,n}.
		\]
		Then
		\[
		u_1S_{\sigma,n}+u_2R_{H,n}=Q_n\bigl(A_n(u)\bigr).
		\]
		By lemma \ref{lem:lowH-opf},
		\[
		\frac{\norm{A_n(u)}_{\op}}{\norm{A_n(u)}_{\F}}\to0,
		\]
		so the Gaussian quadratic-form CLT applies to every fixed linear combination. Moreover, bilinearity and lemma \ref{lem:lowH-trace} give
		\[
		\frac1n\tr\bigl(A_n(u)^2\bigr)
		\to u_1^2\sigma^2\frac{\widetilde K_{\sigma\sigma}(\theta)}{2\pi}
		+u_1u_2\sigma^3\frac{\widetilde K_{\sigma H}(\theta)}{2\pi}
		+u_2^2\frac{\sigma^4}{4}\frac{\widetilde K_{HH}(\theta)}{2\pi}.
		\]
		Since
		\[
		\Var\bigl(Q_n(A_n(u))\bigr)=2\tr\bigl(A_n(u)^2\bigr),
		\]
		the limiting variance is exactly $u^\top \mathcal I_{\sigma,H}^{(0)}(\theta)u$. The Cram\'er--Wold device concludes the proof.
	\end{proof}
	
	\subsection*{D.3. The drift score and the complete three parameter limit}
	
	\begin{lemma}\label{lem:lowH-alpha}
		Assume $0<H<1/2$. Then
		\[
		\frac{S_{\alpha,n}}{\sqrt{n\Delta_n}}\dto \mathcal{N}\Bigl(0,\frac{1}{2\alpha}\Bigr),
		\]
		and the mixed traces satisfy
		\begin{equation}\label{eq:lowH-cross-orders}
			\tr(\widetilde C_{\sigma,n}A_{\alpha,n})=O(n\Delta_n),
			\qquad
			\tr(\widetilde D_{H,n}A_{\alpha,n})=O\bigl(n\Delta_n|\log\Delta_n|\bigr).
		\end{equation}
		In particular,
		\[
		\tr(\widetilde C_{\sigma,n}A_{\alpha,n})=o\bigl(\sqrt n\sqrt{n\Delta_n}\bigr),
		\qquad
		\tr(\widetilde D_{H,n}A_{\alpha,n})=o\bigl(\sqrt n\sqrt{n\Delta_n}\bigr).
		\]
	\end{lemma}
	
	\begin{proof}
		The argument is the same as in the main text, but now the $(\sigma,H)$-block is normalized by $\sqrt n$. First,
		\[
		\tr(A_{\alpha,n}^2)=\frac{n}{2\pi}\int_{-\pi}^{\pi}g_{\alpha,n}(\lambda)^2\,\dd\lambda+o(n\Delta_n),
		\]
		with $\sup_\lambda |g_{\alpha,n}(\lambda)|\le C$ and
		\[
		g_{\alpha,n}(\Delta_n u)\to -\frac{2\alpha}{\alpha^2+u^2}.
		\]
		Hence
		\[
		\int_{-\pi}^{\pi}g_{\alpha,n}(\lambda)^2\,\dd\lambda
		\sim \Delta_n\int_{\R}\frac{4\alpha^2}{(\alpha^2+u^2)^2}\,\dd u
		=\Delta_n\frac{2\pi}{\alpha},
		\]
		so $\tr(A_{\alpha,n}^2)\sim n\Delta_n/\alpha$. Because $\norm{A_{\alpha,n}}_{\op}=O(1)$, the quadratic-form CLT yields
		\[
		\frac{S_{\alpha,n}}{\sqrt{n\Delta_n}}\dto N\Bigl(0,\frac1{2\alpha}\Bigr).
		\]
		
		For the mixed traces, combine the localization bound \eqref{eq:alpha-envelope-subcritical} with \eqref{eq:lowH-envelope}. Then
		\[
		\int_{-\pi}^{\pi}|\widetilde c_n(\lambda)g_{\alpha,n}(\lambda)|\,\dd\lambda
		\le C\int_{-\pi}^{\pi}\frac{1}{1+(\lambda/\Delta_n)^2}\,\dd\lambda
		=O(\Delta_n),
		\]
		and
		\[
		\int_{-\pi}^{\pi}|\widetilde d_n(\lambda)g_{\alpha,n}(\lambda)|\,\dd\lambda
		\le C\int_{-\pi}^{\pi}\frac{1+|\log|\lambda||}{1+(\lambda/\Delta_n)^2}\,\dd\lambda
		=O\bigl(\Delta_n|\log\Delta_n|\bigr).
		\]
		Since $\widetilde c_n,\widetilde d_n\in L^2([{-}\pi,\pi])$ uniformly and $g_{\alpha,n}$ is bounded, the same trace theorem used in lemma \ref{lem:lowH-trace} therefore gives \eqref{eq:lowH-cross-orders}. After division by $\sqrt n\sqrt{n\Delta_n}=n\Delta_n^{1/2}$, the two bounds become $\Delta_n^{1/2}$ and $\Delta_n^{1/2}|\log\Delta_n|$, both tending to zero.
	\end{proof}
	
	\begin{proposition}\label{prop:lowH-full}
		Assume $0<H<1/2$. Then
		\[
		\begin{pmatrix}
			S_{\sigma,n}/\sqrt n\\[1mm]
			R_{H,n}/\sqrt n\\[1mm]
			S_{\alpha,n}/\sqrt{n\Delta_n}
		\end{pmatrix}
		\dto \mathcal{N}\Biggl(0,
		\begin{pmatrix}
			\mathcal I_{\sigma,H}^{(0)}(\theta) & 0\\
			0 & \dfrac{1}{2\alpha}
		\end{pmatrix}
		\Biggr).
		\]
		Equivalently,
		\[
		\diag(\sqrt n,\sqrt n,\sqrt{n\Delta_n})^{-1}
		\begin{pmatrix}
			1&0&0\\
			-\sigma\log\Delta_n&1&0\\
			0&0&1
		\end{pmatrix}
		\begin{pmatrix}
			S_{\sigma,n}\\ S_{H,n}\\ S_{\alpha,n}
		\end{pmatrix}
		\dto \mathcal{N}\Biggl(0,
		\begin{pmatrix}
			\mathcal I_{\sigma,H}^{(0)}(\theta) & 0\\
			0 & \dfrac{1}{2\alpha}
		\end{pmatrix}
		\Biggr).
		\]
	\end{proposition}
	
	\begin{proof}
		Combine proposition \ref{prop:lowH-block} and lemma \ref{lem:lowH-alpha}. The normalized cross-covariances vanish by lemma \ref{lem:lowH-alpha} and the Cram\'er--Wold device gives the three dimensional convergence.
	\end{proof}
	
	\begin{remark}
		We intentionally stop at score asymptotic normality in Appendices~C and~D. Once the transformed score vectors above are available, the LAN expansion follows by the same perturbative matrix argument as in Section~5: one expands the covariance perturbation, identifies the deterministic quadratic term through the trace limits, and observes that in the lower regimes no projection is needed beyond removing the explicit $\sigma\log\Delta_n$ term from the raw $H$-score.
	\end{remark}

\end{document}